\documentclass[aos]{imsart}
\usepackage{natbib}

\usepackage[utf8]{inputenc}
\usepackage{titlesec}
\usepackage{titletoc}
\usepackage{graphicx}
\usepackage{stmaryrd}
\usepackage{amsmath}
\usepackage{bbm}

\usepackage{varioref} 
\usepackage[unicode,bookmarks,colorlinks=true,citecolor=black,linkcolor=black]{hyperref}
\usepackage[all]{hypcap}

\usepackage{amstext}
\usepackage{amsbsy}
\usepackage{amssymb}
\usepackage{amsthm}
\usepackage{dsfont}
\usepackage{newtxtext}
\usepackage[varg]{newtxmath}
\usepackage{eucal}
\usepackage{mathrsfs}
\usepackage{url}
\usepackage[usenames]{xcolor}
\renewcommand{\today}{December 31, 2017}
\newcommand{\ds}{\displaystyle}

\renewcommand{\mathbb}{\mathds}

\renewcommand{\C}[1]{\mathcal{#1}}

\newcommand{\ov}[1]{\overline{#1}}
\newcommand{\wt}[1]{\widetilde{#1}}
\newcommand{\wh}[1]{\widehat{#1}}
\newcommand{\B}[1]{\mathds{#1}}

\newcommand{\ud}{\mathrm{d}}

\newcommand{\HS}{\mathrm{HS}}

\DeclareMathOperator{\Tr}{\mathbf{Tr}}

\newtheorem{thm}{Theorem}[section]
\newtheorem{prop}[thm]{Proposition}

\newtheorem{lemma}[thm]{Lemma}
\newtheorem{cor}[thm]{Corollary}

\theoremstyle{remark}
\newtheorem{rmk}{Remark}[section]

\newcommand{\myeq}[1]{{\rm (\ref{#1})} on page \pageref{#1}}
\allowdisplaybreaks

\begin{document}

\pagestyle{plain}
\renewcommand{\thefootnote}{\fnsymbol{footnote}}
\begin{center}
{\large \sc Dimension-free PAC-Bayesian bounds for matrices, 
vectors, and linear least squares regression.  
}\\[12pt]
{\sc Olivier Catoni 
\footnote{CREST --  CNRS, UMR 9194, Université Paris Saclay, 
France; 
e-mail:  olivier.catoni@ensae.fr} and Ilaria Giulini
\footnote{ Laboratoire de Probabilités et Modèles Aléatoires, 
Université Paris Diderot, France; e-mail: giulini@math.univ-paris-diderot.fr}
}\\[12pt]
{\small \it \today }\\[12pt]
\begin{minipage}{0.8\textwidth}
{\small {\sc Abstract: } 
This paper is focused on dimension-free PAC-Bayesian 
bounds, under weak polynomial moment assumptions, 
allowing for heavy tailed sample distributions. 
It covers the estimation of the mean of a vector 
or a matrix, with applications to least squares 
linear regression. Special efforts are devoted 
to the estimation of the Gram matrix, due to its 
prominent role in high-dimension data analysis. 
\\[1ex]
\\[1ex]
}
\end{minipage} 
\end{center}

\section{Introduction}

The subject of this paper is to discuss dimension-free 
PAC-Bayesian bounds for matrices and vectors. 
It comes after \cite{CatoniGram} and \cite{Giulini01}, 
the first paper discussing dimension dependent bounds 
and the second one dimension-free bounds, under 
a kurtosis like assumption about the data distribution. 
Here, in contrast, we envision even weaker 
assumptions, and focus on dimension-free bounds only.

Our main objective is 
the estimation of the mean of a random vector and of a random matrix. 
Finding sub-Gaussian estimators for the mean of a non necessarily sub-Gaussian 
random vector has been the subject 
of much research in the last few years, with important contributions 
from \cite{JolyLugoOl}, \cite{LugoMen} and \cite{Minsker}.
While in \cite{JolyLugoOl} the statistical error bound still has 
a residual dependence on the dimension of the ambient space, 
in \cite{LugoMen} this dependence is removed, for an estimator 
of the median of means type. However, this estimator is not 
easy to compute and the bound contains large constants.
We propose here another type of estimator, that can be seen as a 
multidimensional extension of \cite{Cat10}. It provides a nonasymptotic
confidence region with the same diameter (including the values of 
the constants)
as the Gaussian concentration inequality stated 
in equation (1.1) of \cite{LugoMen}, although in our case, the confidence 
region is not necessarily a ball, but still a convex set. 
The Gaussian bound concerns the estimation 
of the expectation of a Gaussian random vector by the mean of 
an i.i.d. sample, whereas in our case, we only assume that the 
variance is finite, a much weaker hypothesis.  

In \cite{Minsker2} the question of estimating the mean of a random matrix 
is addressed. The author uses exponential matrix inequalities in order 
to extend \cite{Cat10} to matrices and to control the operator norm 
of the error. In the bounds 
at confidence level $1 - \delta$, the complexity term 
is multiplied by $\log(\delta^{-1})$. 
Here, we extend \cite{Cat10} using PAC-Bayesian bounds to measure complexity, 
and define an estimator with a bound where  the term $\log(\delta^{-1})$ 
is multiplied 
by some directional variance term only, and not the complexity 
factor, that is larger. 

After recalling in Section \ref{sec:2} the PAC-Bayesian inequality that 
will be at the heart of many of our proofs, we 
deal successively with the estimation of a random 
vector (Section \ref{sec:3}) and of a random matrix (Section \ref{sec:4}). 
Section \ref{sec:6} is devoted to the estimation of the Gram matrix, due to its 
prominent role in multidimensional data analysis. In Section \ref{sec:7}
we introduce some applications to least squares regression.

\section{Some well known PAC-Bayesian inequality}\label{sec:2}

This is a preliminary section, where we state the PAC-Bayesian inequality 
that we will use throughout this paper to 
obtain deviation inequalities holding uniformly
with respect to some parameter. 

Consider a random variable $X \in \C{X}$ and
a measurable parameter space $\Theta$. Let 
$\mu  \in \C{M}_+^1(\Theta)$ be a probability measure on $\Theta$  
and $f : \Theta \times \C{X} \rightarrow \B{R}$ a bounded measurable function.  
For any other probability measure $\rho$ on $\Theta$, 
define the Kullback divergence function $\C{K}(\rho, \mu)$ as 
usual by the formula
\[ 
\C{K}(\rho, \mu) = 
\begin{cases}
\ds \int \log \biggl( \ds \frac{\ud \rho}{\ud \mu} \biggr) 
\, \ud \rho, & \rho \ll 
\mu, \\ 
+ \infty, & \text{ otherwise.}
\end{cases}
\] 
Let $(X_1, \dots, X_n)$ be $n$ independent copies of $X$. 
\begin{prop}
\label{prop:2.1}
For any $\delta \in ]0, 1[$, with probability at least $1 - \delta$, 
for any probability measure $\rho \in \C{M}_+^1(\Theta)$, 
\[
\frac{1}{n} \sum_{i=1}^n  \int f(\theta, X_i) \ud \rho(\theta)
\leq \int 
\log \Bigl[ \B{E} \Bigl( \exp \bigl( f(\theta, X) \bigr) \Bigr) \Bigr] 
\, \ud \rho(\theta) + \frac{\C{K}(\rho, \mu) + \log(\delta^{-1})}{n}.
\]
\end{prop}
\begin{proof}
It is a consequence of equation (5.2.1) page 159 of \cite{Cat01b}.
Indeed, let us recall the identity
\[ 
\log \biggl( \int \exp \bigl( h(\theta) \bigr) \, \ud \mu(\theta) \biggr) 
 = \sup_{\rho} \biggl\{ \int h(\theta) \, \ud \rho(\theta) - \C{K}(\rho, \mu)
\biggr\},
\] 
where $h$ may be any bounded measurable function (extensions to unbounded 
$h$ are possible but will not be required in this paper), and where 
the supremum in $\rho$ is taken on all probability measures on 
the measurable parameter space $\Theta$. The proof may be found 
in \cite[page 159]{Cat01b}. Combined with Fubini's 
lemma, it yields
\begin{multline*}
\B{E} \Biggl\{ \exp \sup_{\rho} \Biggl[ \int \Biggl( \sum_{i=1}^n  
f(\theta, X_i) - n \log \Bigl[ \B{E} \Bigl( \exp \bigl( f(\theta, X) 
\bigr) \Bigr) \Bigr] \Biggr) \, \ud \rho(\theta) - \C{K}(\rho, \mu) 
\Biggr] \Biggr\} \\ 
= \B{E} \Biggl\{ \int \exp \Biggl( \sum_{i=1}^n 
f(\theta, X_i) - n \log \Bigl[ \B{E} \Bigl( \exp \bigl( f(\theta, X) \bigr) 
\Bigr) \Bigr] \Biggr) \, \ud \mu(\theta) \Biggr\} \\
= \int \B{E} \exp \Biggl( \sum_{i=1}^n 
f(\theta, X_i) - n \log \Bigl[ \B{E} \Bigl( \exp \bigl( f(\theta, X) \bigr) 
\Bigr) \Bigr] \Biggr) \, \ud \mu(\theta) \\ 
= \int \prod_{i=1}^n \Biggl[ \frac{\ds \B{E} \Bigl( \exp \bigl( f(\theta, X_i) 
\bigr) \Bigr)}{\ds \B{E} \Bigl( \exp \bigl( f(\theta, X) \bigr) \Bigr)} 
\Biggr] \, \ud \mu(\theta) = 1.
\end{multline*}
Since $\B{E}(\exp(W)) \leq 1$ implies that 
\[
\B{P}
\bigl(W \geq \log (\delta^{-1}) \bigr) = \B{E} 
\Bigl( \B{1} \bigl[ \delta \exp(W) \geq 1 \bigr] \Bigr) \leq 
\B{E} \bigl( \delta \exp(W)\bigr) \leq \delta, 
\]
we obtain the desired result, considering 
\[ 
W = \sup_{\rho} \Biggl[ \int \Biggl( \sum_{i=1}^n  
f(\theta, X_i) - n \log \Bigl[ \B{E} \Bigl( \exp \bigl( f(\theta, X) 
\bigr) \Bigr) \Bigr] \Biggr) \, \ud \rho(\theta) - \C{K}(\rho, \mu) 
\Biggr]. 
\] 
\end{proof}

\section{Estimation of the mean of a random vector } 
\label{sec:3}

Let $X \in \B{R}^d$ be a random vector and let $(X_1, \dots, X_n)$ be $n$ independent copies of $X$. 
In this section, we will estimate the mean $\B{E}(X)$ 
and obtain dimension-free non-asymptotic bounds for the estimation 
error.  

Let $\ds \B{S}_d = \bigl\{ \theta \in \B{R}^d \, : \, \lVert \theta \rVert 
= 1 \bigr\}$ be the unit sphere of $\B{R}^d$ and let $I_d$ 
be the identity matrix of size $d \times d$. 
Let $\rho_{\theta} = \C{N} \bigl( \theta, \beta^{-1} I_d \bigr) $ 
be the normal distribution centered at $\theta \in \B{R}^d$, 
whose covariance matrix is $\beta^{-1} I_d$, where $\beta$ is 
a positive real parameter.

Instead of estimating directly the mean vector $\B{E}(X)$, 
our strategy will be rather to estimate its component $\langle \theta, 
\B{E}(X) \rangle$ in each direction $\theta \in \B{S}_d$ of the unit 
sphere.  
For this, we introduce the estimator of $\langle \theta, \B{E} (X) \rangle$ defined as 
\[ 
\C{E} (\theta) = \frac{1}{n \lambda } \sum_{i=1}^n 
\int \psi \bigl( \lambda \langle \theta', X_i \rangle \bigr) 
\, \ud \rho_{\theta}(\theta'),  \qquad \theta \in \B{S}_d, \ \lambda > 0,
\]  
where $\psi$ is the symmetric influence function 
\begin{equation}
\label{eq:1}
\psi(t) = \begin{cases} 
t - t^3/6, &  - \sqrt{2} \leq t \leq \sqrt{2} \\  
2 \sqrt{2} / 3, &  t > \sqrt{2}\\
- 2 \sqrt{2} / 3, &  t < - \sqrt{2}
\end{cases}
\end{equation}
and where the positive constants 
$\lambda$ and $\beta$ will be chosen afterward.

As stated in the following lemma, we chose this influence function
because it is close to the identity in a neighborhood of 
zero and is such that $\exp \bigl( \psi(t) \bigr)$ is bounded by polynomial
functions.  
\begin{lemma}
\label{lem:3.1}
For any $t \in \B{R}$, 
\[ 
- \log \bigl( 1 - t + t^2 / 2 \bigr) \leq \psi(t) \leq 
\log \bigl( 1 + t + t^2 / 2 \bigr).
\] 
\end{lemma}
\begin{proof} Put $f(t) = \log \bigl( 1 + t + t^2 / 2 \bigr)$. 
Remark that $\ds f'(t) = \frac{1 + t}{1 + t + t^2 /2 }$ for $t \in \B{R}$ 
and that $\psi'(t) = 1 - t^2 / 2$ for $t \in [- \sqrt{2}, \sqrt{2}]$. 
As $\psi(0) = f(0) = 0$ and  
\begin{gather*}
\bigl[ f'(t) - \psi'(t) \bigr] \bigl( 1 + t + t^2/2 \bigr) = 
\frac{t^3(2 - t)}{4}, \\
\begin{aligned}
\psi'(t) & \leq f'(t), & 0 & \leq t \leq \sqrt{2},\\
\psi'(t) & \geq f'(t), & - \sqrt{2} & \leq t \leq 0,
\end{aligned}
\end{gather*}
proving that 
\[
\psi(t) \leq f(t), \qquad -\sqrt{2} \leq t \leq \sqrt{2}.
\]
Since $f$ is increasing on $[\sqrt{2}, + \infty[$ 
and decreasing on $ ] - \infty, - \sqrt{2}]$, 
while $\psi$ is constant on these two intervals, 
the above inequality can be extended to all $t \in \B{R}$.  
From the symmetry $\psi(-t) = - \psi(t)$, we deduce 
the converse inequality  
\[
-f(-t) \leq \psi(t), \qquad t \in \B{R}
\] 
that ends the proof. 
\end{proof}
Since 
$\lambda \langle \theta', X_i \rangle$ follows a normal distribution 
with mean $\lambda \langle \theta, X_i \rangle$ 
and standard deviation $\lambda \beta^{-1/2} \lVert X_i \rVert$, 
and since the influence function $\psi$ is piecewise polynomial, 
the estimator $\C{E}$
can be computed explicitly in terms of the standard normal 
distribution function. This is done in the following lemma. 

\begin{lemma}\label{lem:3.2}
Let $W \sim \C{N}(0,1)$ be a standard Gaussian real valued random variable. 
For any $m \in \B{R}$ and any $\sigma \in \B{R}_+$, define 
\[ 
\varphi(m, \sigma) = \B{E} \bigl[ \psi \bigl( m + \sigma W \bigr) \bigr]. 
\] 
The function $\varphi$ can be computed as 
\[ 
\varphi(m, \sigma) = m \bigl( 1 - \sigma^2 / 2 \bigr) - m^3 / 6 + 
r(m,\sigma), 
\] 
where, introducing $F(a) = \B{P}(W \leq a )$, $a \in \B{R}$, the correction 
term $r$ is 
\begin{multline*}
r(m, \sigma) = \frac{2 \sqrt{2}}{3} \Biggl[ 
F \biggl( \frac{ - \sqrt{2} + m}{\sigma} \biggr) - F \biggl( 
\frac{- \sqrt{2} - m}{\sigma} \biggr) \Biggr]  \\ 
- \bigl( m - m^3 / 6 \bigr) \Biggl[ F \biggl( \frac{- \sqrt{2} + m}{\sigma} 
\biggr) + F \biggl( \frac{- \sqrt{2} - m}{\sigma} \biggr) \Biggr]\\ 
+ \sigma \frac{\bigl( 1 - m^2 / 2 \bigr)  }{\sqrt{2 \pi}} 
\Biggl[ \exp \Biggl( - \frac{1}{2} \biggl( 
\frac{ \sqrt{2} + m}{\sigma} \biggr)^2 \Biggr)
- \exp \Biggl( - \frac{1}{2} \biggl( \frac{ \sqrt{2} - m}{\sigma} 
\biggr)^2 \Biggr) \Biggr] \\    
+ \frac{m \sigma^2}{2} \Biggl\{ 
F \biggl( \frac{- \sqrt{2} - m}{\sigma} \biggr) + 
F \biggl( \frac{- \sqrt{2} + m}{\sigma} \biggr) \\ + \frac{1}{\sqrt{2 \pi}}
\Biggl[ \frac{( \sqrt{2} + m)}{\sigma} \exp \Biggl[ 
- \frac{1}{2} \biggl( \frac{\sqrt{2} + m}{\sigma} \biggr)^2 \Biggr] 
+ \frac{(\sqrt{2} - m)}{\sigma} \exp 
\Biggl[ - \frac{1}{2} \biggl( \frac{\sqrt{2} - m}{\sigma} \biggr)^2 
\Biggr] \Biggr\} \\ 
+ \frac{\sigma^3}{6 \sqrt{2 \pi}} \Biggl\{ \Biggl[ \biggl( \frac{\sqrt{2} 
- m}{\sigma} \biggr)^2 + 2 \Biggr] \exp \Biggl[ - \frac{1}{2} 
\biggl( \frac{\sqrt{2} - m}{\sigma} \biggr)^2 \Biggr] \\ - 
\Biggl[ \biggl( \frac{\sqrt{2} + m}{\sigma} \biggr)^2 + 2 \Biggr] 
\exp \Biggl[ - \frac{1}{2} 
\biggl( \frac{\sqrt{2} + m}{\sigma} \biggr)^2 \Biggr] \Biggr\}.
\end{multline*}
\end{lemma}

\vskip2mm
\noindent
Remark that the correction term is small when $\lvert m \rvert $ is small 
and $\sigma$ is small, 
since 
\[ 
F(-t) \leq \min \Biggl\{ \frac{1}{t \sqrt{2 \pi}}, \frac{1}{2} \Biggr\} 
\exp \biggl( - \frac{t^2}{2} \biggr), \qquad t \in \B{R}_+.
\] 
\begin{proof}
The proof of this lemma is a simple computation,
based on the expression 
\begin{multline*}
\psi(t) = \biggl( t - \frac{t^3}{6} \biggr) \bigl[ 
\B{1} ( t \leq \sqrt{2} ) - \B{1} ( t \leq - \sqrt{2} ) \bigr] \\ 
+ \frac{ 2 \sqrt{2}}{3} \bigl[ 1 - \B{1} ( t \leq \sqrt{2} ) - \B{1} ( 
t \leq - \sqrt{2} ) \bigr], \qquad t \in \B{R}, 
\end{multline*}
on the identities
\begin{align*}
\B{E} \bigl[ \B{1} ( W \leq a ) \bigr] & = F(a), \\
\B{E} \bigl[ \B{1} ( W \leq a ) W \bigr] & = - \frac{1}{\sqrt{2 \pi}} 
\exp \biggl( - \frac{a^2}{2} \biggr), \\ 
\B{E} \bigl[ \B{1} ( W \leq a ) W^2 \bigr] & = F(a) - \frac{a}{\sqrt{2 \pi}} 
\exp \biggl( - \frac{a^2}{2} \biggr), \\ 
\B{E} \bigl[ \B{1}( W \leq a ) W^3 \bigr] & = 
- \frac{(a^2 + 2) }{\sqrt{2 \pi}} \exp \biggl( - \frac{a^2}{2} \biggr),
\end{align*}
and on the fact that $F(-t) = 1 - F(t)$. 
\end{proof}
Accordingly, the estimator $\C{E}$
can be computed as 
\begin{multline*}
\C{E}(\theta)  = \frac{1}{n \lambda} \sum_{i=1}^n \varphi \Bigl( \lambda \langle \theta, 
X_i \rangle, \lambda \beta^{-1/2} \lVert X_i \rVert \Bigr) 
\\ = \frac{1}{n} \sum_{i=1}^n \langle \theta, X_i \rangle 
\biggl( 1 - \frac{\lambda^2 \lVert X_i \rVert^2}{2 \beta} \biggr) 
- \frac{\lambda^2 \langle \theta, X_i \rangle^3}{6} + 
r \Bigl( \lambda \langle \theta, X_i \rangle, \lambda \beta^{-1/2} 
\lVert X_i \rVert \Bigr). 
\end{multline*}
\subsection{Estimation without centering}
\begin{prop}
\label{prop:2.3}
Assume that 
\begin{align*}
& \B{E} \bigl( \lVert X \rVert^2 \bigr) = 
\Tr \bigl[ \B{E} \bigl( X X^{\top} \bigr) \bigr] \leq T < \infty \\ 
\text{ and } \qquad & \sup_{\theta \in \C{S}} \B{E} 
\bigl( \langle \theta,   
X \rangle^2 \bigr) 
\leq v \leq T < \infty, 
\end{align*}
where $T$ and $v$ are two known constants and
where $\C{S} \subset \B{S}_d$ is an arbitrary symmetric 
subset of the unit sphere, meaning that if $\theta \in \C{S}$ then $-\theta \in \C{S}$.
Choose any confidence parameter $\delta \in ]0,1[$ and set 
the constants $\lambda$ and $\beta$ used in the definition 
of the estimator $\C{E}$ to 
\begin{align*}
\lambda & = \sqrt{\frac{2 \log(\delta^{-1})}{n v}}, \\ 
\beta & = \sqrt{n T} \lambda = \sqrt{\frac{ 2 T \log(\delta^{-1})}{v}}.
\end{align*}
{\small \sc Non asymptotic confidence region:} With probability at least $ 1 - \delta $, 
\[
\sup_{\theta \in \C{S}} \bigl\lvert \C{E} ( \theta) -  
\langle \theta, \B{E}(X) \rangle \bigr\rvert \leq \sqrt{\frac{T}{n}} 
+ \sqrt{\frac{2 v \log(\delta^{-1})}{n}}.
\]
Consider an estimator $\wh{m} \in \B{R}^d$ of $\B{E}(X)$  satisfying 
\[ 
\sup_{\theta \in \C{S} } \bigl\lvert 
\C{E}(\theta) - \langle \theta, \wh{m} \rangle \bigr\rvert 
\leq \sqrt{\frac{T}{n}} + \sqrt{\frac{2 v \log(\delta^{-1})}{n}}. 
\]
With probability at least $1 - \delta$, such a vector exists and  
\begin{multline*}
\sup_{\theta \in \C{S}} \bigl\lvert \langle \theta, \wh{m} - \B{E}(X) \rangle 
\lvert 
\leq \sup_{\theta \in \C{S}} \bigl\lvert \C{E}(\theta) - \langle 
\theta, \wh{m} \rangle \bigr\rvert + \sqrt{\frac{T}{n}} 
+ \sqrt{\frac{2 v \log(\delta^{-1})}{n}}
\\ \leq 
2 \sqrt{\frac{T}{n}} + 2 \sqrt{\frac{2 v \log(\delta^{-1})}{n}}. 
\end{multline*}
\end{prop}
\begin{rmk}
In particular in the case when $\C{S} = \B{S}_d$ is the whole unit sphere, 
we obtain with probability at least $1 - \delta$ the bound  
\[
\lVert \wh{m} - \B{E}(X) \rVert = \sup_{\theta \in \B{S}_d} 
\bigl\langle \theta, \wh{m} - \B{E}(X) \bigr\rangle 
\leq 2 \Biggl( \sqrt{\frac{T}{n}} + \sqrt{\frac{2 v \log(\delta^{-1})}{n}}
\; \Biggr).
\]
By choosing $\wh{m}$ as the middle of a diameter of the confidence region,
we could do a little better and replace the factor $2$ in this bound 
by a factor $\sqrt{3}$. 
\end{rmk}
\begin{proof}
According to the PAC-Bayesian inequality of Proposition \vref{prop:2.1}, 
with probability at least
$1 - \delta$, for any $\theta \in \C{S}$, 
\[
\C{E} (\theta) \leq 
\frac{1}{\lambda} 
\int \log \Bigl[ \B{E} \Bigl( \exp \psi \bigl( \lambda \langle 
\theta', X \rangle \bigr) \Bigr) \Bigr] 
\, \ud \rho_{\theta}(\theta') + 
\frac{\C{K}(\rho_{\theta}, \rho_0) + \log(\delta^{-1})}{n \lambda}.
\]
We can then use the polynomial approximation of $\exp(\psi(t))$ 
given by Lemma \vref{lem:3.1}, remarking that $\C{K}(\rho_{\theta}, 
\rho_0) = \beta / 2$ and that $\log(1 + z) \leq z$, to deduce that  
\begin{multline*}
\C{E}(\theta) \leq \B{E} \bigl( \langle \theta, X \rangle \bigr)
+ \frac{\lambda}{2} \int \B{E} \bigl( \langle \theta', X \rangle^2 \bigr) 
\, \ud \rho_{\theta}(\theta') + \frac{ \beta 
+ 2 \log(\delta^{-1})}{2 n \lambda} 
\\ = \B{E} \bigl( \langle \theta, X \rangle \bigr) + \frac{\lambda}{2} 
\biggl[  \B{E} \bigl( \langle \theta, X \rangle^2 \bigr) 
+ \frac{\B{E}(\lVert X \rVert^2)}{ \beta } \biggr] + \frac{\beta + 
2 \log(\delta^{-1})}{2 n \lambda} 
\\ \leq \B{E} \bigl( \langle \theta, X \rangle \bigr) 
+ \frac{\lambda}{2} \bigl( v + T / \beta \bigr)  
+ \frac{\beta + 2 \log(\delta^{-1})}{2 n \lambda}
\\ = \langle \theta, \B{E}(X) \rangle + \sqrt{\frac{T}{n}} + 
\sqrt{\frac{2 v \log(\delta^{-1}) }{n}}.
\end{multline*}
We conclude by considering both $\theta \in \C{S}$ and $- \theta 
\in \C{S}$ 
to get the reverse inequality, using the assumption that 
$\C{S}$ is symmetric and remarking that $\C{E}(- \theta) = - \C{E}(\theta)$. \\
The existence with probability $1 - \delta$ of $\wh{m}$ satisfying 
the required inequality is granted by the fact that on the event defined 
by the above PAC-Bayesian inequality, the expectation $\B{E}(X)$ 
belongs to the confidence region that, as a result, cannot 
be empty. 
\end{proof}

\subsection{Centered estimate}
The bounds in the previous section are simple, but they 
are stated in terms of uncentered moments of order two where 
we would have expected a variance. 
In this section, we explain how to deduce centered bounds 
from the uncentered bounds of the previous section, through the use of 
a sample splitting scheme.\\[1ex]
Assume that 
\begin{align*}
& \B{E} \bigl( \lVert X - \B{E}(X) \rVert^2 \bigr) \leq \ov{T} < \infty,\\
\text{and } & \sup_{\theta \in \B{S}_d} \B{E} \bigl( \langle \theta, 
X - \B{E}(X) \rangle^2 \bigr) \leq \ov{v} \leq \ov{T} < \infty, 
\end{align*}
where $\ov{v}$ and $\ov{T}$ are known constants. 
Remark that when these bounds hold, the bounds  
\begin{equation}\label{eq:vT}
v = \ov{v} + \lVert \B{E}(X) \rVert^2 \text{ and } 
T = \ov{T} + \lVert \B{E}(X) \rVert^2
\end{equation} 
hold in the previous section. 
Assume that we know also some bound $b$ such that
\[
\bigl\lVert \B{E}(X) \bigr\rVert^2 \leq b.
\]
Split the sample in two parts $(X_1, \dots, X_k)$ 
and $(X_{k+1}, \dots, X_n)$. 
Use the first part to construct an estimator $\wt{m}$ of 
$\B{E}(X)$ as described in Proposition \vref{prop:2.3}, 
choosing $\C{S} = \B{S}_d$. 
According to this proposition and by equation \eqref{eq:vT}, with probability at least
$1 - \delta$, 
\[ 
\lVert \wt{m} - \B{E}(X) \rVert \leq 2 \sqrt{\frac{\ov{T}+b}{k}}
+ 2 \sqrt{\frac{2 (\ov{v} + b) \log(\delta^{-1})}{k}} = 
\sqrt{\frac{A}{k}}, 
\] 
where we have put $A = 4 \Bigl( \sqrt{\ov{T} + b} + 
\sqrt{2 (\ov{v} + b) \log(\delta^{-1})} \; \Bigr)^2$.\\[1ex]
We then construct an estimator $\C{E}(\theta)$ of $\langle \theta, 
\B{E}(X) \rangle$, $\theta \in \B{S}_d$, 
built as described in 
Proposition \ref{prop:2.3}, based on the sample $(X_{k+1} - \wt{m}, 
\dots, X_n - \wt{m})$ and on the constants $\ov{T} + A/k$ 
and $\ov{v} + A/k$.
With probability at least $1 - 2 \delta$, 
\[ 
\sup_{\theta \in \B{S}_d} \bigl\lvert \C{E}(\theta) - \langle \theta, 
\B{E}(X) \rangle \bigr\rvert \leq B_{n,k} = \sqrt{ \frac{\ov{T} + A/k}{n-k}} 
+ \sqrt{\frac{2 \bigl( \ov{v} + A / k \bigr) \log(\delta^{-1})}{n - k}},
\] 
and we can, if needed, deduce from $\C{E}(\theta)$ an estimator $\wh{m}$ 
such that with probability at least $1 - 2 \delta$, 
\[ 
\lVert \wh{m} - \B{E}(X) \rVert \leq 2 B_{n,k}.  
\]  

\vskip2mm
If we want the correction term $A / k$ to behave as a second order term 
when $n$ tends to $\infty$, we can for example take $k = \sqrt{n}$, 
in which case $n - k$ is equivalent to $n$ at infinity, 
so that $B_{n,\sqrt{n}}$ is equivalent to 
\[ 
\sqrt{\frac{\ov{T}}{n}} + \sqrt{\frac{2 \, \ov{v} \log(\delta^{-1})}{n}}.
\]

\vskip2mm
Let us also mention that a simpler estimator, obtained by shrinking the norm 
of $X_i$, is also possible. It comes with a sub-Gaussian deviation bound 
under the slightly stronger hypothesis that 
$\B{E} \bigl( \lVert X \rVert^{p}\bigr) 
< \infty$ for some (non necessarily integer) exponent $p > 2$, 
and is described in \cite{CatGiul2017}.

\section{Mean matrix estimate} 
\label{sec:4}

Let $M \in \B{R}^{p \times q}$ be a random matrix and let 
\linebreak $M_1 , \dots, M_n$ 
be $n$ independent copies of $M$.
In this section, we will provide an estimator for $\B{E}(M)$. 

From the previous section, we already have an estimator $\wh{m}$ of $\B{E}(M)$
with a bounded Hilbert-Schmidt norm $\lVert \wh{m} - \B{E}(M) \rVert_{\HS}$, 
since from the point of view of the 
Hilbert-Schmidt norm, $M$ is nothing but a random vector of 
size $pq$.
Here, we will be interested in another natural norm, 
the operator norm 
\[
\lVert M \rVert_{\infty} = \sup_{\theta \in \B{S}_q}\lVert M \theta \rVert.
\]
Indeed, recalling that
\[
\lVert M \rVert_{\infty} = \sup_{\theta \in \B{S}_q, \xi \in \B{S}_p} 
\langle \xi, M \theta \rangle = \sup_{\xi \in \B{S}_p} \lVert 
M^{\top} \xi \rVert = \sup_{\theta \in \B{S}_q, \xi \in \B{S}_p} 
\Tr \bigl( \theta \xi^{\top} M \bigr),
\]
we see that we can deduce results 
from the previous section on vectors, considering the scalar 
product between matrices
\[ 
\langle M, N \rangle = \Tr \bigl( M^{\top} N \bigr), \quad M, N 
\in \B{R}^{p \times q}, 
\] 
and the part of the unit sphere defined as 
\[ 
\C{S} = \bigl\{ \xi \theta^{\top} \, : \, \xi \in \B{S}_p, \theta \in 
\B{S}_q \bigr\}. 
\] 
Doing so, we obtain in the uncentered case a bound of the form 
\[ 
\lVert \wh{m} - \B{E}(M) \rVert \leq 2 \sqrt{ 
\frac{\B{E} ( \lVert M \rVert_{\HS}^2 ) }{n}} + 2 \sqrt{
\frac{2}{n} \sup_{\xi \in \B{S}_p, \theta \in \B{S}_q} 
\B{E} \bigl( \langle \xi, M \theta \rangle^2 \bigr) \log(\delta^{-1})}.
\] 
We will show in the next section that the second $\delta$-dependent term is satisfactory whereas 
the first $\delta$-independent term can be improved. 

\subsection{Estimation without centering} 

Consider the influence function $\psi$ defined by equation \myeq{eq:1}.

For any $\xi \in \B{R}^p$, let $\nu_{\xi} = \C{N} \bigl( \xi, \beta^{-1} I_p
\bigr)$, where $I_p$ is the identity matrix of size $p \times p$.
In the same way, let $\rho_{\theta} = \C{N} \bigl( \theta, \gamma^{-1} I_q
\bigr)$, $\theta \in \B{R}^q$. 
Consider the estimator of $\langle \xi, \B{E}(M) \, \theta \rangle$ defined 
as 
\[
\C{E}(\xi, \theta) = \frac{1}{\lambda n} \sum_{i=1}^n \int \psi 
\bigl( \lambda \langle \xi', M_i \theta' \rangle
\bigr) \, \ud 
\nu_{\xi}(\xi') \, \ud \rho_{\theta}(\theta'), \qquad \xi \in \B{R}^p, 
\theta \in \B{R}^q. 
\] 

\begin{prop}
For any parameters $\delta \in ]0,1[$, $\lambda, \beta, \gamma
\in ]0, \infty[$,  
with probability at least $1 - \delta$, for any $\xi \in \B{R}^p$ 
and any $\theta \in \B{R}^q$,  
\begin{multline*}
\bigl\lvert \C{E}(\xi, \theta) - \B{E} \bigl( \langle \xi, 
M \theta \rangle \bigr) \bigr\rvert \\ \leq \frac{\lambda}{2} 
\biggl[ \B{E} \bigl( \langle \xi, M \theta \rangle^2 \bigr) 
+ \frac{\B{E} \bigl( \lVert M \theta \rVert^2 \bigr) }{\beta} 
+ \frac{\B{E} \bigl( \lVert M^{\top} \xi \rVert^2 \bigr) }{\gamma} 
+ \frac{ \B{E} \bigl( \lVert M \rVert_{\HS}^2\bigr)}{\beta \gamma} 
\\ + \frac{\beta + \gamma + 2 \log(\delta^{-1})}{2 n \lambda}. 
\end{multline*}
\end{prop}
\begin{proof}
The PAC-Bayesian inequality of Proposition \vref{prop:2.1} 
tells us that  
with probability at least $1 - \delta$,
for any $\xi \in \B{R}^p$ and any $\theta \in \B{R}^q$, 
\begin{multline*}
\C{E} ( \xi, \theta) \leq \lambda^{-1} \int \log \Bigl\{ \B{E} \Bigl[ \exp \Bigl( 
\psi \bigl( \lambda \langle \xi', M \theta' \rangle 
 \bigr) \Bigr) \Bigr] \Bigr\} 
\, \ud \nu_{\xi}(\xi') \ud \rho_{\theta}(\theta')
\\ + \frac{\C{K}(\nu_{\xi}, \nu_0)}{n \lambda} + \frac{\C{K}(\rho_{\theta}, 
\rho_0)}{n \lambda} + \frac{\log(\delta^{-1})}{n \lambda},
\end{multline*}
Using the properties of $\psi$ (Lemma \vref{lem:3.1}) and Fubini's lemma, 
we get 
\[
\C{E}(\xi, \theta) \leq   
\langle \xi, \B{E}(M) \theta \rangle + \frac{\lambda}{2} \B{E} \biggl( \int \langle \xi', M \theta' \rangle^2  
\, \ud \nu_{\xi}(\xi') \ud \rho_{\theta}(\theta') \biggr) + \frac{ \beta + \gamma 
+ 2 \log  (\delta^{-1})}{2n \lambda}. 
\]
As
\[ 
\int \langle \xi', M \theta' \rangle^2 \, \ud \nu_{\xi}(\xi') \ud \rho_{\theta} 
(\theta') =  \langle \xi, M \theta \rangle^2 + \frac{ \lVert 
M \theta \rVert^2}{\beta} + \frac{\lVert M^{\top} \xi \rVert^2}{\gamma} 
+ \frac{ \lVert M \rVert_{\HS}^2}{\beta \gamma}, 
\] 
this concludes the proof.  
\end{proof}

Let us now discuss the question of computing $\C{E}(\xi, \theta)$. 
Remark that, according to Lemma \vref{lem:3.2}, for any $x \in \B{R}^p$, 
\begin{multline*}
\int \psi \bigl( \langle \xi', x \rangle \bigr) \, \ud \nu_{\xi}(\xi') 
= \varphi \bigl( \langle \xi, x \rangle, \beta^{-1/2} \lVert x \rVert \bigr)
\\ = 
\langle \xi , x \rangle - \frac{\langle \xi, x \rangle \, \lVert x \rVert^2}{2 \beta}  
- \frac{ \langle \xi, x \rangle^3}{6} + r \Bigl( \langle \xi, x \rangle , 
\beta^{-1/2} \lVert x \rVert \Bigr). 
\end{multline*}
It is also easy to check that 
\begin{multline*}
\int \lVert M_i \theta' \rVert^2 \, \ud \rho_{\theta}(\theta') = 
\lVert M_i \theta \rVert^2 + \frac{\lVert M_i \rVert_{\HS}^2}{\gamma}, \\
\shoveleft{\int \langle \xi, M_i \theta' \rangle \lVert M_i \theta' \rVert^2 \, \ud 
\rho_{\theta}(\theta') = \langle \xi , M_i \theta \rangle \lVert M_i \theta \rVert^2
+ \frac{1}{\gamma} \langle \xi, M_i \theta \rangle \lVert M_i \rVert_{\HS}^2 
}\\ \shoveright{+ \frac{2}{\gamma} \langle \xi, M_i M_i^{\top} M_i \theta \rangle,} \\
\shoveleft{\text{and }\int 
\langle \xi, M_i \theta' \rangle^3 \, \ud \rho_{\theta}(\theta') = 
\langle \xi, M_i \theta \rangle^3 + \frac{3}{\gamma} 
\langle \xi, M_i \theta \rangle
\lVert M_i^{\top} \xi \rVert^2.} \hfill
\end{multline*}
Consider a standard random vector $W_q \sim \C{N}(0, I_q)$. 
We obtain that 
\begin{multline*}
\int \psi \bigl( \lambda \langle \xi', M_i \theta' \rangle 
\bigr) \, \ud \nu_{\xi} 
(\xi') \ud \rho_{\theta}(\theta') =  \\
\int \Biggl( \lambda \langle \xi, M_i \theta' \rangle - \frac{\lambda^3 
\langle \xi, M_i \theta' \rangle \lVert M_i \theta' 
\rVert^2}{2 \beta} - \frac{\lambda^3}{6} 
\langle \xi, M_i \theta' \rangle^3 \\ 
+ r \Bigl( \lambda \langle \xi, M_i \theta'\rangle, 
\lambda \beta^{-1/2} 
\lVert M_i \theta' \rVert \Bigr) \Biggr) \, \ud \rho_{\theta}(\theta'),
\end{multline*} 
so that 
\begin{multline*}
\C{E}(\xi, \theta) = \frac{1}{n} \sum_{i=1}^n \langle \xi , M_i \theta \rangle 
- \frac{\lambda^2}{6} \langle \xi, M_i \theta \rangle^3 \\  
- \frac{ \lambda^2}{2 \beta} \langle \xi, M_i \theta \rangle 
\lVert M_i \theta \rVert^2
- \frac{\lambda^2}{2 \gamma} \langle \xi, M_i \theta \rangle 
\lVert 
M_i^{\top} \xi \rVert^2 \\ - \frac{ \lambda^2}{2 \beta \gamma} 
\langle \xi, M_i \theta 
\rangle \lVert M_i \rVert^2_{\HS} - \frac{ \lambda^2}{\beta \gamma} 
\langle \xi, M_i M_i^{\top} M_i \theta \rangle 
\\
+ \frac{1}{\lambda} \B{E} \biggl[  r \Bigl( \lambda 
\langle M_i^{\top} \xi,  \theta + \gamma^{-1/2} W_q \rangle, 
\lambda \beta^{-1/2} \lVert M_i (\theta + \gamma^{-1/2} W_q) \rVert \Bigr) 
\biggr]. 
\end{multline*}
The last term is not explicit, since it contains an expectation, 
but should be most of the time a small 
reminder and can be evaluated using a Monte-Carlo numerical scheme. 
This gives a more explicit and efficient method than evaluating 
directly $\C{E}(\xi, \theta)$ using a Monte-Carlo simulation 
for the couple of random variables $(\xi', \theta') \sim 
\nu_{\xi} \otimes \rho_{\theta}$.
\begin{prop}
\label{prop:3.2}
Assume that 
the following finite bounds are known 
\begin{align*}
v & \geq  \sup_{\xi \in \B{S}_p, \theta \in \B{S}_q} \B{E} \bigl( \langle \xi, M 
\theta \rangle^2 \bigr) = \sup_{\xi \in \B{S}_p, \theta \in \B{S}_q} 
\bigl( \xi^{\top} \otimes \xi^{\top} \bigr) 
\B{E} \bigl( M \otimes M \bigr) \bigl( \theta \otimes \theta), \\ 
t & \geq \sup_{\theta \in \B{S}_q} \B{E} \bigl( \lVert M \theta \rVert^2 
\bigr) = \sup_{\theta \in \B{S}_q} \langle \theta, \B{E} \bigl( M^{\top} M 
\bigr) \theta \rangle = \lVert \B{E} \bigl( M^{\top} M \bigr) \rVert_{\infty},\\
u  & \geq \sup_{\xi \in \B{S}_p} \B{E} \bigl( \lVert M^{\top} \xi \rVert^2 \bigr)
= \sup_{\xi \in \B{S}_p} \langle \xi, \B{E} \bigl( M M^{\top} \bigr) 
\xi \rangle = \lVert \B{E} \bigl( M^{\top} M \bigr) \rVert_{\infty},\\
T & \geq \B{E} \bigl( \lVert M \rVert_{\HS}^2 \bigr),
\end{align*}
and choose 
\[
\lambda = \sqrt{ \frac{ \beta + \gamma + 2 \log (\delta^{-1})}{
n \bigl( v + t / \beta + u / \gamma + T / (\beta \gamma) \bigr)}}.
\]
For any values of $\delta \in ]0, 1[$, $\beta, \gamma \in ]0, \infty[$, 
with probability at least $1 - \delta$, for any $\xi \in \B{S}_p$, 
any $\theta \in \B{S}_q$, 
\[ 
\bigl\lvert \C{E}(\xi, \theta) - \langle \xi, \B{E}(M) \theta \rangle  
\bigr\rvert \leq B_n = \sqrt{ \Bigl( v + \frac{t}{\beta} + 
\frac{u}{\gamma} + \frac{T}{\beta \gamma} \Bigr) \ \frac{ 
\beta + \gamma + 2 \log(\delta^{-1})  }{n}}.
\]  
Consider now any estimator $\wh{m}$ of $\B{E}(M)$. With probability at least 
$1 - \delta$, 
\[ 
\lVert \wh{m} - \B{E} (M) \rVert_{\infty} 
\leq \sup_{\xi \in \B{S}_p, \theta \in \B{S}_q} \bigl\lvert 
\C{E}(\xi, \theta) - \langle \xi, \wh{m} \, \theta \rangle \bigr\rvert 
+ B_n.
\] 
In particular, if we choose $\wh{m}$ such that, 
\[ 
\sup_{\xi \in \B{S}_p, \theta \in \B{S}_q} \bigl\lvert 
\C{E}(\xi, \theta) - \langle \xi, \wh{m} \, \theta \rangle \bigr\rvert 
\leq B_n, 
\] 
with probability at least $1 - \delta$, this choice is possible
and 
\[
\lVert \wh{m} - \B{E}(M) \rVert_{\infty} \leq 2 B_n. 
\]
\end{prop} 
\begin{rmk}
In particular, choosing $\ds \beta = \gamma = 2 \max \Biggl\{ 
\frac{(t+u)}{v}, \sqrt{\frac{T}{v}} \; \Biggr\}$,  
we get 
\[
B_n \leq \sqrt{\frac{2 v }{n} \Biggl( 2 \log (\delta^{-1}) 
+ 4 \max \Biggl\{ \frac{t + u}{v}, \sqrt{\frac{T}{v}} \; \Biggr\}
\; \Biggr)}. 
\]
The bound $B_n$ is of the type 
$\ds \sqrt{\frac{2 v \bigl[ \C{C} + \log(\delta^{-1}) \bigr]}{n}}$, with a complexity 
(or dimension) term $\C{C}$ equal to 
\[ 
\C{C} = 4 \max \Biggl\{ \, \frac{t + u}{v}, \sqrt{\frac{T}{v}} \; \Biggr\} + 
\log(\delta^{-1}).  
\] 
\end{rmk}
\begin{rmk}
Let us envision a simple case to compare the 
precision of the bounds in a setting where 
dimension-free and dimension-dependent 
bounds coincide. 
Assume more specifically 
that the entries of the matrix $M$, 
\[
M_{i,j} \quad 1 \leq i \leq p, 1 \leq j \leq q, 
\]
are centered and i.i.d.. 
Assume that
$\sigma = 
\sqrt{\B{E}(M_{i,j}^2)}$ is known, 
and take  
\begin{align*}
v & = \sup_{\xi \in \B{S}_p, \theta \in \B{S}_q} \B{E} 
\bigl( \langle \xi, M \theta \rangle^2 \bigr) = \sigma^2,\\
t & = \sup_{\theta \in \B{S}_q} \B{E} \bigl( \lVert M \theta
\rVert^2 \bigr) = p \sigma^2, \\ 
u & = \sup_{\xi \in \B{S}_p} \B{E} \bigl( \lVert M^{\top} \xi \rVert^2 \bigr) 
= q \sigma^2, \\ 
T & = \B{E} \bigl( \lVert M \rVert_{\HS}^2 \bigr) = p q \sigma^2.
\end{align*}
Choosing $\beta = \gamma = 2 (p+q)$, we get a complexity term equal to 
\[ 
\C{C} = 4 ( p + q ) +  \log(\delta^{-1}),
\] 
whereas the bound of the previous section made for vectors has a complexity
factor equal to $pq$. 
\end{rmk}
\subsection[Controlling both errors]{Controlling both the operator norm error and 
the Hilbert-Schmidt error}

There are situations where it is desirable to control both 
$\lVert \wh{m} - \B{E}(M) \rVert_{\infty}$ 
and $\lVert \wh{m} - \B{E}(M) \rVert_{\HS}$. 
To do so we can very easily combine Propositions \vref{prop:2.3}
and Proposition \vref{prop:3.2}, since these two propositions are based on 
the construction of confidence regions.

More precisely, first consider $M\in \B{R}^{p \times q}$ as a vector and use the scalar product 
\[ 
\langle \theta, M \rangle_{\HS} = \Tr \bigl( \theta^{\top} M \bigr), \quad 
\theta \in \B{R}^{p \times q}. 
\] 
Applying Proposition \vref{prop:2.3}, we can build an estimator 
$\C{E}_{\HS}(\theta)$ such that with probability at least $1 - \delta$, 
\[
\sup_{\theta \in \B{R}^{p \times q}, \lVert \theta \rVert_{\HS} = 1} 
\bigl\lvert \, \C{E}_{\HS}(\theta) - \Tr \bigl(
\theta^{\top} \B{E}(M) \bigr) \bigr\rvert 
\leq A_n = \sqrt{\frac{T}{n}} + \sqrt{\frac{2 v \log(\delta^{-1})}{n}}.
\]
On the other hand, we can also apply Proposition \vref{prop:3.2} and build an estimator 
$\C{E}(\xi, \theta), \xi \in \B{S}_p, \theta \in \B{S}_q$, 
such that with probability at least $1 - \delta$, 
\[
\sup_{\xi \in \B{S}_p, \theta \in \B{S}_q} 
\bigl\lvert \, \C{E}(\xi, \theta) - \langle \xi, \B{E}(M) \theta \rangle 
\bigr\rvert 
\leq B_n = \sqrt{\frac{2v}{n} \Biggl( 
2 \log(\delta^{-1}) + 4 \max \Biggl\{ \frac{t + u}{v}, \sqrt{\frac{T}{v}} 
\; \Biggr\} \; \Biggr)}. 
\]

\begin{prop}
\label{prop:3.3.2}
Consider a matrix $\wh{m}$ such that 
\begin{gather*}
\sup_{\theta \in \B{R}^{p \times q}, \lVert \theta \rVert_{\HS} = 1} 
\bigl\lvert \C{E}_{\HS}(\theta) - \Tr \bigl( \theta^{\top} 
\wh{m} \bigr) \bigr\rvert \leq A_n \\
\text{and } \quad \sup_{\xi \in \B{S}_p, \theta \in \B{S}_q} 
\bigl\lvert \C{E}(\xi, \theta) - \langle \xi, \wh{m} \, \theta \rangle 
\bigr\rvert \leq B_n. 
\end{gather*}
Combining Propositions \ref{prop:2.3} and \ref{prop:3.2} 
shows that, with probability at least $1 - 2 \delta$, 
such a matrix $\wh{m}$ exists and satisfies both
\[ 
\lVert \wh{m} - \B{E}(M) \rVert_{\HS} \leq 2 A_n \quad \text{ and } \quad 
\lVert \wh{m} - \B{E}(M) \rVert_{\infty} \leq 2 B_n. 
\] 
\end{prop}

\vskip2mm
Remark that $B_n$ is typically smaller than $A_n$ as expected
in interesting large dimension situations.

\subsection{Centered estimator}
As already done in the case of the estimation of the mean of a random vector, we deduce in this section centered bounds from the uncentered bounds of the previous sections, using sample splitting.

Put $m = \B{E}(M)$ and $\ov{M} = M - m$. Assume that we know finite 
constants $ \ov{v},  \ov{t},  \ov{u},  \ov{T}$ such that 
\begin{align*}
& \sup_{\xi \in \B{S}_p, \theta \in \B{S}_q} \B{E} 
\bigl( \langle \xi, \ov{M} \theta \rangle^2 \bigr) \leq \ov{v} 
< \infty, \\ 
& \sup_{\theta \in \B{S}_q} \B{E} \bigl( \lVert \ov{M} \theta 
\rVert^2 \bigr) \leq \ov{t} < \infty, \\ 
& \sup_{\xi \in \B{S}_p} \B{E} \bigl( \lVert \ov{M}^{\top} \xi \rVert^2
\bigr) \leq \ov{u} < \infty, \\ 
& \B{E} \bigl( \lVert \ov{M} \rVert^2_{\HS} \bigr) \leq \ov{T} < \infty. 
\end{align*}
When this is true, we can take for the previous uncentered constants 
\[ 
v = \ov{v} + \lVert m \rVert_{\infty}^2, \quad
t = \ov{t} + \lVert m \rVert_{\infty}^2, \quad
u = \ov{u} + \lVert m \rVert_{\infty}^2, \quad
T = \ov{T} + \lVert m \rVert_{\HS}^2.
\] 

In view of this, it is suitable to assume 
that we also know some finite 
constants $b$ and $c$ such that
\[ 
\lVert m \rVert_{\infty}^2 \leq b \quad \text{ and } \quad 
\lVert m \rVert_{\HS}^2 \leq c. 
\] 

\vskip2mm
As we see that the Hilbert-Schmidt norm $\lVert m \rVert_{\HS}$
comes into play, we will use the combined 
preliminary estimate provided by Proposition \vref{prop:3.3.2}.

Given an i.i.d. matrix sample $(M_1, \dots, M_n)$, 
first use 
$(M_1, \dots, M_k)$ to build a preliminary estimator $\wt{m}$ 
as described in Proposition \ref{prop:3.3.2}. 
With probability at least $1 - \delta/2$, 
\begin{align*}
\lVert \wt{m} - m \rVert_{\HS} & \leq \sqrt{\frac{A}{k}} \quad \text{ and } 
\quad \lVert \wt{m} - m \rVert_{\infty} \leq \sqrt{\frac{B}{k}}, \\ 
\text{where } A & = 4 \biggl( \sqrt{2 (\ov{v} + b) \log(4/\delta)} + \sqrt{\ov{T} + c}  
\biggr)^2\\ 
\text{and } B & = 8 (\ov{v} + b) \Biggl( 
2 \log(4/\delta) + 4 \max \Biggl\{ \, \frac{\ov{t} + \ov{u} + 2b}{\ov{v} + b}, 
\biggl( \frac{\ov{T} + c}{\ov{v} + b}\biggr)^{1/2} \, \Biggr\} \Biggr).
\end{align*}
Then use the sample $(M_{k+1} - 
\wt{m}, \dots, M_n - \wt{m})$ to build an estimator \linebreak $\C{E}(\xi, 
\theta), \; \xi \in \B{S}_p, 
\; \theta \in \B{S}_q,$
based on the construction described in 
Proposition \vref{prop:3.2}, at confidence level $1 - \delta / 2$. 
It is such that with probability 
at least $1 - \delta$, 
\begin{multline*}
\bigl\lvert \C{E}(\xi, \theta) - \langle \xi, m \, \theta \rangle 
\bigr\rvert \leq C_{n,k} \\ = \sqrt{ \frac{2(\ov{v} + B/k)}{n-k} 
\Biggl( 2 \log(2/\delta) + 4 \max \Biggl\{ \, \frac{\ov{t} + \ov{u} 
+ 2 B/k}{\ov{v} + B/k}, \biggl( \frac{\ov{T} + A/k}{\ov{v} + B/k}
\biggr)^{1/2} \, \Biggr\} \Biggr)}.
\end{multline*}
If we choose for instance $k = \sqrt{n}$, we obtain that 
\[
C_{n,\sqrt{n}} \underset{n \rightarrow \infty}{\sim}
\sqrt{\frac{2 \, \ov{v}}{n} \biggl( 2 \log (2/ \delta) + 4 
\max \Biggl\{ \, \frac{\ov{t} + \ov{u}}{\ov{v}}, \, \biggl( \frac{\ov{T}}{\ov{v}}
\biggr)^{1/2} \, \Biggr\} \Biggr)}.
\]

\section{Adaptive estimators}
\label{sec:5}

The results presented in the previous sections 
assume that there exist 
known upper bounds for some quantities as $\B{E}(\lVert X \rVert^2)$ in the case of a mean vector estimate
or $\B{E}(\lVert M \rVert_{\HS}^2)$ in the matrix case. 
Here we would like to adapt to these quantities, in the case when 
those bounds are not known.

To do so, we will use an asymmetric influence function $\psi
: \B{R}_+ \longrightarrow \B{R}_+$
defined on the positive real line only as 
\begin{equation}
\label{eq:2}
\psi(t) = \begin{cases}
t - t^2 / 2, & 0 \leq t \leq 1, \\ 
1/2, & 1 \leq t.
\end{cases}
\end{equation}
\begin{lemma}
\label{lem:5.1}
For any $t \in \B{R}_+$, 
\[ 
- \log( 1 - t + t^2) \leq \psi(t) \leq \log(1 + t).
\] 
\end{lemma}
\begin{proof}
Let us put $f(t) = - \log(1 - t + t^2)$ and $g(t) = \log(1 + t)$.
Remark that $f(0) = g(0) = \psi(0) = 0$. Remark also that for any 
$t \in [0,1]$, 
\begin{gather*}
f'(t) = \frac{1 - 2t}{1 - t + t^2}, \quad \psi'(t) - f'(t) = 
\frac{t^2(2 - t)}{1 - t + t^2} \leq 0 , \\ \text{ and } \quad g'(t) = 1 - \frac{t}{1+t} \geq \psi'(t)
= 1 - t.
\end{gather*}
As on the interval $[1, \infty[$, $f$ is decreasing, $g$ is increasing
and $\psi$ is constant, this proves the lemma.  
\end{proof}

Similarly to the previous case, considering a standard Gaussian 
real valued random variable $W \sim \C{N}(0,1)$, we can introduce 
the function
\[ 
\varphi(m, \sigma) = \B{E} \bigl\{ \psi \bigl[ (m + \sigma W)_+ \bigr] \bigr\},
\] 
where $t_+ = \max \, \bigl\{ \, t, \, 0 \, \bigr\}$, and explicitly compute $\varphi$ as
\begin{multline*}
\varphi(m, \sigma) 
= \biggl( m - \frac{m^2}{2} - \frac{\sigma^2}{2} 
\biggr) \biggl[ 1 - F \biggl( \frac{-1+m}{\sigma} \biggr) 
- F \biggl( - \frac{m}{\sigma} \biggr) \biggr] 
+ \frac{1}{2} F \biggl( \frac{-1 + m}{\sigma} \biggr)
\\ + \frac{\sigma \bigl( 1 - m /2 \bigr)}{\sqrt{2 \pi}} \exp \biggl( - \frac{m^2}{2 \sigma^2} \biggr)
- \frac{\sigma ( 1 - m)}{2 \sqrt{2 \pi}} \exp \biggl( - \frac{(1-m)^2}{
2 \sigma^2} \biggr), 
\end{multline*}
using the expression 
\[ 
\psi(t_+) = \Bigl( t - t^2/2 \Bigr) \Bigl[ \B{1} ( t \leq 1 ) 
- \B{1}( t \leq 0) \Bigr] + \frac{1}{2} \Bigl[ 1 - \B{1}(t \leq 1) \Bigr], \qquad t \in \B{R}.
\] 
\subsection{Estimation of the mean of a random vector}
Consider a discrete set $\Lambda$ of values of $\lambda$
and a probability measure $\mu$ on $\Lambda$, to 
be chosen more precisely later on. Let $\beta$ be some 
positive parameter that we will also choose later and put
as previously $\rho_{\theta} = \C{N}(\theta, \beta^{-1} I_d)$. 
Define for any $\theta \in \B{S}_d$ 
\begin{align*}
& \C{E}_+(\theta) = \sup_{\lambda \in \Lambda} \frac{1}{n \lambda} \sum_{i=1}^n 
\int \psi \bigl( \lambda \langle \theta', X_i \rangle_+ \bigr) 
\, \ud \rho_{\theta}(\theta') - \frac{\beta + 2 \log \bigl(
\delta^{-1} \mu(\lambda)^{-1} \bigr) }{2 \lambda n}, \\ 
& \C{E}_-(\theta) = \sup_{\lambda \in \Lambda} \frac{1}{n \lambda} 
\sum_{i=1}^n \int \psi \bigl( \lambda \langle \theta', X_i \rangle_- \bigr) 
\, \ud \rho_{\theta}(\theta') - \frac{ \beta + 2 \log \bigl( \delta^{-1} 
\mu(\lambda)^{-1} \bigr)}{2 \lambda n}, \\
& \text{and } \C{E} (\theta) = \C{E}_+(\theta) - \C{E}_-(\theta).  
\end{align*}
Thoughtful readers may wonder why we introduce $\lambda$ in this way 
and do not use instead $\rho_{\lambda \theta}$, to get a uniform 
result in $\lambda \theta$ in one shot, 
without introducing the discrete set $\Lambda$. 
It is because this option would produce the entropy factor 
$\ds \frac{\lambda \beta}{2n}$ instead of $\ds \frac{\beta}{2n \lambda}$, 
requiring a value of $\beta$ depending on unknown moments 
of the distribution of $X$. 

According to the PAC-Bayesian inequality of Proposition \vref{prop:2.1}, 
with probability at least 
$1 -  2 \delta$, 
\begin{multline*}
\int \B{E} \bigl( \langle \theta', X \rangle
_+ \bigr) \, \ud \rho_{\theta}(\theta') \\ - \inf_{\lambda \in \Lambda}
\Biggl\{ \lambda \int \B{E} \bigl( \langle \theta', X \rangle_+^2
\bigr) \, \ud \rho_{\theta}(\theta')  + \frac{ \beta + 2 \log \bigl( 
\delta^{-1} \mu(\lambda)^{-1} \bigr)}{\lambda n} \Biggr\} \\ \leq  \C{E}_+(\theta) \leq \int \B{E} \bigl( 
\langle \theta', X \rangle_+ \bigr)
\, \ud \rho_{\theta}(\theta'). 
\end{multline*}
More precisely, to obtain the above inequalities, we have used a union bound with 
respect to $\lambda \in \Lambda$, starting from the
fact that, when we replace the infimum in $\lambda$ in the 
previous equation with a fixed value of $\lambda \in \Lambda$, 
it holds with probability at least $1 - 2 \mu(\lambda) \delta$. 
Since 
\[
\int f(- \theta') \, \ud \rho_{\theta} (\theta') = \int f(\theta') 
\, \ud \rho_{- \theta}( \theta'), 
\]
this implies also that 
\begin{multline*}
\int \B{E} \bigl( \langle \theta', X \rangle
_- \bigr) \, \ud \rho_{\theta}(\theta') \\ - \inf_{\lambda \in \Lambda}
\Biggl\{ \lambda \int \B{E} \bigl( \langle \theta', X \rangle_-^2
\bigr) \, \ud \rho_{\theta}(\theta')  + \frac{ \beta + 2 \log \bigl( 
\delta^{-1} \mu(\lambda)^{-1} \bigr)}{\lambda n} \Biggr\} \\ \leq  \C{E}_-(\theta) \leq \int \B{E} \bigl( 
\langle \theta', X \rangle_- \bigr)
\, \ud \rho_{\theta}(\theta'). 
\end{multline*}
Therefore, with probability at least $1 - 2 \delta$, 
\begin{multline*}
- B_-(\theta) = - \inf_{\lambda \in \Lambda} \Biggl\{ \lambda \int \B{E} \bigl( 
\langle \theta', X \rangle_+^2 \bigr) \, \ud \rho_{\theta}(\theta') 
+ \frac{\beta + 2 \log \bigl( \delta^{-1} \mu(\lambda)^{-1} \bigr)}{
\lambda n} \Biggr\}.
\\ \leq
\C{E}(\theta) - \langle \theta, \B{E}(X) \rangle 
\leq  
B_+(\theta) = \inf_{\lambda \in \Lambda} \Biggl\{ \lambda \int \B{E} \bigl( 
\langle \theta', X \rangle_-^2 \bigr) \, \ud \rho_{\theta}(\theta') 
\\ + \frac{\beta + 2 \log \bigl( \delta^{-1} \mu(\lambda)^{-1} \bigr)}{
\lambda n} \Biggr\}.
\end{multline*}
This defines for $\langle \theta, \B{E} ( X ) \rangle$ a confidence interval 
of length no greater than 
\[ 
B(\theta) = \inf_{\lambda \in \Lambda} \Biggl\{ \lambda \int \B{E} 
\bigl( \langle \theta', X \rangle^2 \bigr) \, \ud \rho_{\theta}(\theta') + 
\frac{2 \beta + 4 \log \bigl( \delta^{-1} \mu(\lambda)^{-1} \bigr)}{\lambda n}
\Biggr\}.
\]  
Unfortunately, neither $B_+(\theta)$, $B_-(\theta)$ nor 
$B(\theta)$ are observable. 
But, nevertheless, we can build an estimator $\wh{m}$ such that 
\[ 
\sup_{\theta \in \B{S}_d} \bigl\{ 
\langle \theta, \wh{m} \rangle - \C{E}(\theta) \bigr\} 
= \inf_{m \in \B{R}^d} \sup_{\theta \in \B{S}_d} \bigl\{ 
\langle \theta, m \rangle 
- \C{E}(\theta) \bigr\}.   
\] 
It satisfies with probability at least $1 - 2 \delta$, 
\begin{multline*}
\lVert \wh{m} - \B{E}(X) \rVert = \sup_{\theta \in \B{S}_d} 
\langle \theta, \wh{m} - \B{E}(X) \rangle \\ \leq \sup_{\theta \in \B{S}_d} 
\bigl\{ \langle \theta, \wh{m} \rangle - \C{E}(\theta) \bigr\}  + \sup_{\theta 
\in \B{S}_d} \bigl\{ \C{E}(\theta) - \langle \theta, \B{E}(X) \rangle 
\bigr\} \leq 2 \sup_{\theta \in \B{S}_d} 
\bigl\{ \C{E}(\theta) - \langle \theta, \B{E}(X) \rangle \bigr\}
\\ \leq 2 \sup_{\theta} B_+(\theta) \leq \sup_{\theta \in \B{S}_d} 
\inf_{\lambda \in \Lambda} 2 \lambda \int \B{E}\bigl( \langle \theta', 
X \rangle^2 \bigr) \, \ud \rho_{\theta}(\theta') + \frac{ 
2 \beta + 4 \log \bigl(\delta^{-1} \mu(\lambda)^{-1}\bigr)}{
\lambda n} \\ = \sup_{\theta \in \B{S}_d} \inf_{\lambda \in \Lambda} 
2 \lambda \biggl( \B{E} \bigl( \langle \theta, X \rangle^2 \bigr) + 
\frac{\B{E} \bigl( \lVert X \rVert^2 \bigr)}{\beta}\biggr) 
+ \frac{ 2 \beta + 4 \log \bigl( 
\delta^{-1} \mu(\lambda)^{-1} \bigr)}{\lambda n}.
\end{multline*}

\begin{lemma}
Let us choose $\beta = 2 \log \bigl( \delta^{-1} \bigr)$
and put $v = \sup_{\theta \in \B{S}_d} \B{E} \bigl( \langle 
\theta, X \rangle^2 \bigr)$ and $T = \B{E}(\lVert X \rVert^2)$. 
With probability at least $1 - 2 \delta$, 
\begin{multline*}
\lVert \wh{m} - \B{E}(X) \rVert \leq \inf_{\lambda \in \Lambda} 
B(\lambda), \\ \text{where } B(\lambda) =  \Biggl\{ 
2 \lambda \biggl( v + \frac{T}{2 \log(\delta^{-1})} \biggr) + \frac{
8 \log (\delta^{-1}) + 4 \log \bigl( \mu(\lambda)^{-1} \bigr)}{\lambda n}
\Biggr\}.
\end{multline*}
\end{lemma}
To turn this lemma into an explicit bound, we need now to choose $\Lambda$ 
and $\mu \in \C{M}_+^1 \bigl( \Lambda \bigr)$. 
Consider for some real parameters $\sigma > 0$ and $\alpha > 1$,
\[ 
\Lambda = \Biggl\{ \, \frac{\alpha^k}{\sigma \sqrt{n}} \, : \, 
k \in \B{Z} \, \Biggr\}.
\] 
For any $\ds \lambda_k = \frac{\alpha^k}{\sigma \sqrt{n}} \in \Lambda$, put 
\[ 
\mu(\lambda) = 
\begin{cases}
\ds \frac{1}{2 \bigl(\lvert k \rvert + 1 \bigr) \bigl( \lvert k \rvert + 2 \bigr)}, 
& k \neq 0,\\
1/2, & k = 0, 
\end{cases}
\] 
and remark that 
\[
\mu(\lambda_k) \geq \frac{1}{2 \bigl( \lvert k \rvert + 2 \bigr)^2},
\qquad k \in \B{Z}. 
\]
Put also 
\[ 
\lambda_* = \sqrt{\frac{\ds 4 \log(\delta^{-1})}{\ds n \biggl(
v + \frac{T}{2 \log(\delta^{-1})} \biggr)}}.
\] 
The bound $B(\lambda)$ appearing in the previous lemma can be 
written as
\[
B(\lambda) = 4 \sqrt{ \frac{2}{n} \Bigl(2 v \log(\delta^{-1}) + T \Bigr)}
\; 
\Biggl[ \cosh \Biggl( \log \biggl( \frac{\lambda}{\lambda_*} \biggr) \Biggr) + 
\frac{\log \bigl( \mu(\lambda)^{-1} \bigr)}{4 \log(\delta^{-1})} 
\frac{\lambda*}{\lambda} \Biggr].
\]
Since $\log(\lambda_k) = k \log (\alpha) - \log(\sigma) - \log(n)/2$, there exists $k_* \in \B{Z}$ such that
\[ 
\bigl\lvert \, \log \bigl( \lambda_{k_*} / \lambda_* \bigr) \, \bigr\rvert 
\leq \log (\alpha) / 2, 
\] 
so that 
\[ 
\lvert k_* \rvert \, \leq \bigl\lvert \, \log \bigl( \sigma \lambda_* \sqrt{n}
\, \bigr) 
\bigr\rvert / \log(\alpha)  
+ 1 / 2.
\] 
Therefore 
\[
\inf_{\lambda \in \Lambda} B(\lambda) \leq B(\lambda_{k_*}) 
\leq 4 C \sqrt{\frac{2}{n} \Bigl( 2 v \log(\delta^{-1}) + T \Bigr)} 
\]
where the constant $C$ is equal to 
\[
C = \cosh \Biggl( \, \frac{\log(\alpha)}{2} \, \Biggr) 
+ \frac{\sqrt{\alpha}}{2 \log(\delta^{-1})}  
\log  \Biggl[  \frac{1}{\sqrt{2}\log(\alpha)} \Biggl\lvert \log \Biggl( 
\frac{\ds
2 v \log(\delta^{-1}) + T }{
8 \sigma^2 \log(\delta^{-1})^2} 
\Biggr) \Biggr\rvert   + \frac{5}{\sqrt{2}} \, \Biggr].
\]
We see that the constant $\sigma^2$ can be interpreted as our best guess  
of the ratio 
\[ 
\frac{2 v \log(\delta^{-1}) + T}{8 \log \bigl(\delta^{-1})^2}.
\] 
However, this guess may be very loose without harming the constant 
$C$ too much. Indeed, to give an example,  
if we choose $\alpha = e$
and we assume that we made an error of magnitude $10^6$ 
on the choice of $\sigma^2$, compared to the optimal guess, 
we get 
\[
C \leq \cosh(1/2) + \frac{\exp(1/2)}{2 \log(\delta^{-1})} 
\log \Biggl[ \frac{1}{\sqrt{2}} \log ( 10^6 )  + \frac{5}{\sqrt{2}} 
\biggr]  \leq 
1.13 + \frac{2.2}{\log(\delta^{-1})},  
\]so that if we work at the confidence level corresponding to $\delta = 1/100$, 
we obtain that $C \leq 1.6$. In brief, the message is 
that $C$ is typically between one and two. 

\subsection{Adaptive estimation of the mean of a random matrix}
We consider here the same framework as in Section \vref{sec:3}.
Let $M \in \B{R}^{p \times q}$ be a random matrix and
 $(M_1, \dots, M_n)$ be a sample made of $n$ 
independent copies of $M$.

Using the asymmetric 
influence function $\varphi$ defined by 
equation \myeq{eq:2}, given $\xi \in \B{S}_p, \theta \in \B{S}_q$, we define the estimators 
\begin{multline*} 
\C{E}_+(\xi, \theta) = 
\sup_{\lambda \in \Lambda} 
\Biggl\{ \frac{1}{\lambda n} \sum_{i=1}^n 
\int \psi \bigl[ \lambda \langle \xi', M_i \theta' \rangle_+ \bigr] \, 
\ud \nu_{\xi}(\xi') \, \ud \rho_{\theta}(\theta') \\ 
\shoveright{- \frac{\beta + \gamma + 
2 \log \bigl( \delta^{-1} \mu(\lambda)^{-1} \bigr)}{
2 \lambda n}  \Biggr\}, 
} 
\\ 
\shoveleft{\C{E}_-(\xi, \theta) = - \C{E}_+( - \xi, \theta)
= \sup_{\lambda \in \Lambda} 
\Biggl\{ \frac{1}{\lambda n} \sum_{i=1}^n 
\int \psi \bigl[ \lambda \langle \xi', M_i \theta' \rangle_- \bigr] \, 
\ud \nu_{\xi}(\xi') \, \ud \rho_{\theta}(\theta')} \\ 
- \frac{\beta + \gamma + 
2 \log \bigl( \delta^{-1} \mu(\lambda)^{-1} \bigr)}{
2 \lambda n}  \Biggr\}, 
\end{multline*}
and 
$\C{E}(\xi, \theta) = \C{E}_+(\xi, \theta) - \C{E}_-(\xi, \theta). $

\begin{lemma}
With probability at least $1 - 2 \delta$, for any $\xi \in \B{S}_p, 
\theta \in \B{S}_q$, 
\begin{multline*}
- \inf_{\lambda \in \Lambda} \Biggl\{
\lambda \int \B{E} \bigl( \langle \xi', M \theta' \rangle_+^2 \bigr)
\, \ud \nu_{\xi}(\xi') \ud \rho_{\theta}(\theta') 
+ \frac{ \beta + \gamma + 2 \log \bigl( \delta^{-1} \mu(\lambda)^{-1} \bigr)
}{n \lambda} \Biggr\} 
\\ \leq \C{E}_+ \bigl(\xi, \theta) - \int \B{E} \bigl( \langle \xi', 
M \theta' \rangle_+ \bigr) \, \ud \nu_{\xi}(\xi') \ud 
\rho_{\theta}(\theta') \leq 0, 
\end{multline*}
so that 
\begin{multline*}
- B_+ ( - \xi, \theta ) \leq \C{E}(\xi, \theta) - \langle \, \xi, \B{E}(M) \, 
\theta \, \rangle \leq 
B_+(\xi, \theta) \\ =  \inf_{\lambda \in \Lambda} 
\Biggl\{ \lambda \int \B{E} \bigl( \langle \xi', M \theta' \rangle_-^2 
\bigr) \, \ud \nu_{\xi}(\xi') \ud \rho_{\theta}(\theta') 
+ \frac{\beta + \gamma + 2 \log \bigl( \delta^{-1} \mu(\lambda)^{-1} 
\bigr)}{\lambda n} \Biggr\} \\ 
\leq \inf_{\lambda \in \Lambda} \Biggl\{ \lambda \Biggl[ 
\B{E} \bigl( \langle \xi, M \theta \rangle^2 \bigr) 
+ \frac{\B{E} \bigl( \lVert M \theta \rVert^2 \bigr)}{\beta} + 
\frac{\B{E} \bigl( \lVert M^{\top} \xi \rVert^2 \bigr)}{\gamma} + 
\frac{\B{E} \bigl( \lVert M \rVert_{\HS}^2 \bigr)}{\beta \gamma}  
\Biggr] \\ + \frac{ \beta + \gamma + 2 \log \bigl( \delta^{-1} 
\mu(\lambda)^{-1} \bigr)}{\lambda n} \Biggr\}. 
\end{multline*}
\end{lemma}
Choose $\beta = \gamma = 2 \chi \log (\delta^{-1})$, with $\chi >0$. Let 
\[
\Lambda = \biggl\{ \, \lambda_k = 
\frac{\alpha^k}{\sigma \sqrt{n}}\, : \, k \in \B{Z} \, \biggr\}
\]
and 
\[ 
\mu(\lambda_k)^{-1} \leq  2 \bigl( \lvert k \rvert + 2 \bigr)^2,
\] 
as in the previous section. 
Put 
\begin{align*}
v & = \B{E} \bigl( \langle \xi, M \theta \rangle^2 \bigr) 
\; \leq \sup_{\xi \in \B{S}_p, \theta \in \B{S}_q }  \B{E} \bigl( \langle \xi, M \theta \rangle^2 \bigr)= v_*, \\ 
t & = \B{E} \bigl( \lVert M \theta \rVert^2 \bigr) 
\; \leq \sup_{ \theta 
\in \B{S}_q }  \B{E} \bigl( \lVert M \theta \rVert^2 \bigr)= t_*, \\ 
u & = \B{E} \bigl( \lVert M^{\top} \xi \rVert^2 \bigr) 
\; \leq \sup_{\xi \in \B{S}_p} \B{E} \bigl( \lVert M^{\top} \xi \rVert^2 
\bigr)= u_*, \\ 
T & = \B{E} \bigl( \lVert M \rVert_{\HS}^2 \bigr), \\ 
\ell & = \log(\delta^{-1}), \\ 
\lambda_*^2 & = \frac{2 (1 + \chi) \ell^2}{\ds n \biggl(  \ell v 
+ \frac{t + u}{\chi} + \frac{T}{ \ell \chi^2} \biggr)}.
\end{align*}
Remark that, in a similar way to the case of a vector treated in 
the previous section,  
\begin{multline*}
B_+(\xi, \theta) 
= \inf_{\lambda \in \Lambda} \lambda \biggl( v + \frac{t + u}{\chi \ell} 
+ \frac{T}{ \chi^2 \ell^2} \biggr) + \frac{ 2 (\chi + 1) \ell 
+ 2 \log(\mu(\lambda)^{-1})}{ \lambda n} 
\\ \leq \inf_{\lambda \in \Lambda} 
 2 \sqrt{ \frac{2(1 + \chi)}{n} \biggl( \ell v + \frac{t + u}{\chi} 
+ \frac{T}{ \ell \chi^2} \biggr)} \Biggl\{ 
\cosh \biggl[ \log \biggl( \frac{\lambda}{\lambda_*} \biggr) 
\biggr] + \frac{\lambda_* \log \bigl( \mu(\lambda)^{-1} \bigr) }{2 
\lambda (1 + \chi) \ell} \Biggr\}. 
\end{multline*}
Replacing $\lambda_*$ by its value, choosing $\lambda = \lambda_{k_*}$ 
such that $ \bigl\lvert \log( \lambda / \lambda_* ) \bigr\rvert \leq 
\log(\alpha) / 2$ and remarking that 
\[ 
\lvert k_* \rvert \leq \frac{ \bigl\lvert \log \bigl( \sqrt{n} \sigma \lambda_* \bigr) 
\bigr\rvert}{\log(\alpha)} + \frac{1}{2}, 
\]  
we obtain 
\begin{prop}
With probability at least $1 - 2 \delta$, for any $\xi \in \B{S}_p$, 
any $\theta \in \B{S}_q$, 
\begin{multline*}
\bigl\lvert \C{E}(\xi, \theta) - \langle \xi, \B{E}(M) \theta \rangle
\bigr\rvert
\\ \leq B(\xi, \theta) = 2 C \sqrt{ \frac{2(1 + \chi)}{n} \biggl( 
v \log(\delta^{-1}) 
+ \frac{t + u}{ \chi} + \frac{T}{\chi^2 \log(\delta^{-1})} 
\biggr)},  
\end{multline*}
where, using the abbreviation $\ell = \log(\delta^{-1})$,  
\begin{multline*}
C = \cosh \biggl( \frac{\log(\alpha)}{2} \biggr) 
\\ + \frac{\sqrt{\alpha}}{(1 + \chi) \ell } 
\log 
\left\{ \rule[-3.5ex]{0pt}{9ex} \right. \frac{1}{\sqrt{2} \log(\alpha)} 
\left\lvert \rule[-3.5ex]{0pt}{9ex} \right. 
\log 
\left( \rule[-3.5ex]{0pt}{9ex} \right. 
\frac{\ds \rule[-2.5ex]{0pt}{4ex} 
\ell v + \frac{t + u}{\chi} + \frac{T}{\ell \chi^2}}{
\ds \rule{0ex}{3ex} 2 \sigma^2(1 + \chi) \ell^2} 
\left. \rule[-3.5ex]{0pt}{9ex} \right) 
\left. \rule[-3.5ex]{0pt}{9ex} \right\rvert 
+ \frac{5}{\sqrt{2}} 
\left. \rule[-3.5ex]{0pt}{9ex} \right\}.
\end{multline*}
Let us now consider an estimator $\wh{m}$ such that 
\[ 
\sup_{\xi \in \B{S}_p, \theta \in \B{S}_q} 
\bigl\lvert \C{E}(\xi, \theta) - \langle \xi, \wh{m} \, \theta \rangle 
\bigr\rvert 
\leq \inf_{m \in \B{R}^{p \times q}} 
\sup_{\xi \in \B{S}_p, \theta \in \B{S}_q} 
\bigl\lvert \C{E}(\xi, \theta) - \langle \xi, m \,\theta \rangle 
\bigr\rvert. 
\] 
With probability at least $1 - 2 \delta$, 
\[ 
\lVert \wh{m} - \B{E}(M) \rVert_{\infty} \leq 2 
\sup_{\xi \in \B{S}_p, \theta \in \B{S}_q} B(\xi, \theta).
\] 
\end{prop}
Remark that we can bound 
$\sup_{\xi \in \B{S}_p, \theta \in \B{S}_q} B(\xi, \theta)$
by the explicit expression for $B(\xi, \theta)$ where $v$, $t$ 
and $u$ are replaced by their upper bounds $v_*, t_*$, and $u_*$
with respect to $\xi \in \B{S}_p$ and $\theta \in \B{S}_q$.

Remark also that we can weaken the influence of $T$ by choosing 
$\chi > 1$, but that we can reach the optimal bound for 
$\lVert \wh{m} - \B{E}(M) \rVert_{\infty}$
only if we 
know an upper bound for the ratio $T / v_*$. 
Indeed, if we know $T/ v_*$ (or an upper bound of the same order 
of magnitude, up to a constant), we can choose 
\[ 
\chi = \max \Biggl\{ \frac{1}{\log(\delta^{-1})} \sqrt{ \frac{T}{v_*}} , 
1 \Biggr\}. 
\] 
In this case, with probability at least $1 - 2 \delta$, 
\[
\lVert \wh{m}  - \B{E}(M) \rVert_{\infty}  
\\ \leq \frac{8 C}{\sqrt{n}} \sqrt{ v_* \log(\delta^{-1})  
+ t_* + u_* + \sqrt{ v_* T}}.
\]
Most likely we do not know 
\[ 
\sqrt{\frac{T}{v_*}} = \sqrt{\frac{\B{E} \bigl( \lVert M \rVert_{\HS}^2 
\bigr)}{ \sup_{\xi \in \B{S}_p, \theta \in \B{S}_q} 
\B{E} \bigl( \langle \xi, M \theta \rangle^2 \bigr)}}, 
\] 
but we can still choose $\chi$ greater than 
one, to lower the influence of $T = \B{E} \bigl( \lVert M \rVert_{\HS}^2
\bigr)$ in the bound. 

\section{Adaptive Gram matrix estimate}\label{sec:6}

We devote a section to the adaptive estimation of a Gram
matrix, since it is an important subject for applications 
to principal component analysis and to least squares regression. We recall that, 
given a random vector $X \in \B{R}^d$, the Gram matrix of $X$ is defined as
\[
G = \B{E} \bigl( X X^{\top} \bigr) \in \B{R}^{d \times d}.
\]
The general approach of the previous section uses an estimator 
that cannot be computed explicitly without recourse to a Monte Carlo 
sampling algorithm. In the special case of the Gram matrix, 
we will produce an estimator that does not suffer from this 
drawback.

Consequences of what is proved in this section regarding 
robust principal component analysis can easily be drawn
from the method exposed in \cite{Giulini02}. We refer 
to this paper for further details. Consequences regarding 
least squares regression are discussed at the end of this
paper. 

In this section, we will use the asymmetric influence 
function defined by equation \myeq{eq:2}. 
The explicit computation of our estimator however will 
use the modified auxiliary function 
\[
\varphi_2(m,\sigma) = \mathds E \left[ \psi\left( (m+ \sigma W)^2 \right)\right], 
\qquad m \in \B{R}, \sigma \in \B{R}_+,
\]
where $W \sim \C{N}(0,1)$ is a standard 
Gaussian random variable.\\
Observe that it is possible to explicitly compute the function $\varphi_2$ 
in terms of the Gaussian distribution function $F(a) = \B{P}(W \leq a)$.
\begin{lemma}
For any $m \in \B{R}$ and $\sigma \in \B{R}_+$, 
\[
\varphi_2(m, \sigma) = 
m^2 + \sigma^2 - \frac{1}{2} \Bigl( m^4 + 6 m^2 \sigma^2 + 3 \sigma^4 \Bigr) 
+ r_2(m, \sigma), 
\]
where
\begin{multline*}
r_2(m, \sigma) = \frac{1}{2} \Bigl[ (m^2 - 1)^2 + (6m^2 - 2) \sigma^2 
+ 3 \sigma^4 \Bigr] \biggl[ F \biggl( \frac{-1-m}{\sigma} \biggr) 
+ F \biggl( \frac{-1 + m}{\sigma} \biggr) \biggr] \\ + 
\frac{\sigma}{2 \sqrt{2 \pi}} \Bigl[ \sigma^2(3 - 5m) - (1+m)(1-m)^2 \Bigr] 
\exp \biggl( - \frac{(1 + m)^2}{2 \sigma^2} \biggr) \\ 
+ \frac{\sigma}{2 \sqrt{2 \pi}} \Bigl[ \sigma^2(3 + 5m) 
- (1 -m) (1 + m)^2 \Bigr] \exp \biggl( - \frac{(1 -m)^2}{2 \sigma^2} \biggr). 
\end{multline*}
\end{lemma}

\noindent
\begin{proof}
The proof is based on the expression 
\[
\psi(t) = 
t - t^2 / 2 + \B{1}(t \geq 1) (1 - t)^2/2, \qquad t \in \B{R}_+, 
\]
and on the identities
\begin{align*}
& 
\B{E} \Bigl[ \B{1} \bigl( W \leq a \bigr) \Bigr] = F(a)
= 1 - \B{E} \Bigl[ \B{1} \bigl( W \geq a \bigr) \Bigr], 
\\
& 
\B{E} \Bigl[ W \B{1} \bigl( W \leq a \bigr) \Bigr] = -
\frac{1}{\sqrt{2\pi}} \exp \biggl( -\frac{a^2}{2} \biggr)  = 
- \B{E} \Bigl[ W \B{1} \bigl( W \geq a \bigr) \Bigr], 
\\
&  \B{E} \Bigl[  W^2 \B{1} \bigl( W \leq a \bigr) \Bigr] = 
- \frac{a}{\sqrt{2\pi}} \exp \biggl( - \frac{a^2}{2} \biggr)  + F(a)
= 1 - \B{E} \Bigl[ W^2 \B{1} \bigl( W \geq a \bigr) \Bigr], \\ 
& \B{E} \Bigl[ W^3 \B{1} \bigl( W \leq a \bigr) \Bigr] 
= - \frac{a^2+2}{\sqrt{2\pi}} \exp \biggl( - \frac{a^2}{2} \biggr) 
= - \B{E} \Bigl[ W^3 \B{1} \bigl( W \geq a \bigr) \Bigr],  
\\
& 
\B{E} \Bigl[ W^4 \B{1} \bigl( W \leq a \bigr) \Bigr] = - 
\frac{a^3+3a}{\sqrt{2\pi}} \exp \biggl( - \frac{a^2}{2} \biggr)  + 3 F(a)  
= 3 - \B{E} \Bigl[ W^4 \B{1} \bigl( W \geq a \bigr) \Bigr]. 
\end{align*}
Let us put $\ds G(t) = \frac{1}{\sqrt{2 \pi}} \exp \biggl( - 
\frac{t^2}{2} \biggr)$. 
\[
\B{E} \Bigl\{ \psi \bigl[ (m + \sigma W)^2 \bigr] \Bigr\} = 
\B{E} \Bigl[ \bigl( m + \sigma W \bigr)^2 - \frac{1}{2} \bigl( m + \sigma W 
\bigr)^4 \Bigr] + r_2(m, \sigma),
\] where 
\begin{multline*}
2 r_2(m,\sigma) = 
\B{E} \Biggl\{  \Bigl[ (m + \sigma W)^4 - 2 (m + \sigma W)^2  + 1 \Bigr]
\\ \shoveright{\times \biggl[ 
\B{1} \biggl( W \leq \frac{-1-m}{\sigma} \biggr) + \B{1} \biggl( W \geq 
\frac{1 - m}{\sigma} \biggr) \biggr] \Biggr\} }\\ 
= \B{E} \Biggl\{ 
\Bigl[  \bigl( m^2 - 1 \bigr)^2 + 4 m (m^2 - 1) \sigma W + \bigl( 6 
m^2 - 2 \bigr) \sigma^2 W^2 + 4 m \sigma^3 W^3 + \sigma^4 W^4 \Bigr] 
\\ \shoveright{ \times \biggl[ 
\B{1} \biggl( W \leq \frac{-1-m}{\sigma} \biggr) + \B{1} \biggl( W \geq 
\frac{1 - m}{\sigma} \biggr) \biggr] \Biggr\} }
\\ \shoveleft{ 
= \bigl( m^2 - 1 \bigr)^2 \biggl[ F \biggl(  \frac{- 1 - m}{\sigma} 
\biggr) + F \biggl( \frac{ - 1 + m}{\sigma} \biggr) \biggr] }\\ 
+ 4 m (m^2 - 1) \sigma \biggl[ - G \biggl( \frac{-1-m}{\sigma} \biggr) + 
G \biggl( \frac{1 - m}{\sigma} \biggr) \biggr] 
\\ + ( 6 m^2 - 2 ) \sigma^2 \biggl[ \frac{1+m}{\sigma} G \biggl( 
\frac{-1-m}{\sigma} \biggr) + \frac{1-m}{\sigma} G \biggl( \frac{1-m}{\sigma}
\biggr) \\\shoveright{+ F \biggl( \frac{-1-m}{\sigma} \biggr) 
+ F \biggl( \frac{-1 + m}{\sigma} \biggr) \biggr]} \\ 
+ 4 m \sigma^3 \biggl\{ - \biggl[ \biggl( \frac{1+m}{\sigma}   
\biggr)^2 + 2 \biggr] G \biggl( \frac{-1-m}{\sigma} \biggr) + 
\biggl[ \biggl( \frac{1-m}{\sigma} \biggr)^2 + 2 \biggr] G 
\biggl( \frac{-1 + m}{\sigma} \biggr) \biggr\}
\\ 
+ \sigma^4 \biggl\{ \biggl[ \biggl( \frac{1+m}{\sigma} \biggr)^3 + 
\frac{3(1+m)}{\sigma} \biggl] G \biggl( \frac{-1-m}{\sigma} \biggr) 
+ \biggl[ \biggl( \frac{1-m}{\sigma} \biggr)^3 +  
\frac{3(1-m)}{\sigma} \biggr] G \biggl( \frac{1-m}{\sigma} \biggr) 
\\ + 3 \biggl[ F \biggl( \frac{-1-m}{\sigma} \biggr) + F \biggl( 
\frac{-1 + m}{\sigma} \biggr) \biggr] \biggr\},
\end{multline*}
so that 
\begin{multline*}
r_2(m, \sigma) = \frac{1}{2} \Bigl[ (m^2-1)^2 + \bigl( 6 m^2 - 2 \bigr) \sigma^2 + 
3 \sigma^4 \Bigr] \biggl[ F \biggl( \frac{-1-m}{\sigma} \biggr) 
+ F \biggl( \frac{-1 + m}{\sigma} \biggr) \biggr] \\
+ \frac{1}{2} \sigma \Bigl[ \sigma^2 \bigl( 3 - 5 m \bigr) - (1 + m) (1 - m)^2 \Bigr] 
G \biggl( \frac{-1-m}{\sigma} \biggr) \\ + 
\frac{1}{2} \sigma \Bigl[ \sigma^2 \bigl( 3 + 5 m \bigr) - (1-m)(1+m)^2 \Bigr] G
\biggl( \frac{1 - m}{\sigma} \biggr). 
\end{multline*}
\end{proof}
Observe now that, when $\theta'$ is distributed according to  
$ \rho_{\theta} = \C{N}(\theta, \beta^{-1} I_d)$, the real valued 
random variable $\langle \theta', x \rangle$ 
is Gaussian with mean $\langle \theta, x \rangle$ 
and standard deviation $ \lVert x \rVert / \sqrt{\beta}$. 
Thus
we can state the following.
\begin{lemma}\label{lem3.1}
For any $\theta, x  \in \B{R}^d$, 
\[
\int \psi(\langle \theta', x \rangle^2 ) \, \mathrm d \rho_\theta(\theta') = 
\varphi_2 \left(\langle \theta, x \rangle,\frac{\lVert x \rVert}{\sqrt{\beta}}  \right). 
\]
\end{lemma}

Introduce 
\[
A_{\lambda,\beta}(\theta, x) = \varphi_2 \left( \lambda^{1/2} 
\langle \theta, x \rangle,  \frac{\lVert x \rVert}{\sqrt{\beta}} \right) 
- \log\left( 1+ \frac{\|x\|^2}{\beta}  \right),
\]
where $\lambda \in \B{R}_+$ is a constant modifying the norm of
$\theta$. 
Next proposition provides some upper and lower bounds. 
\begin{prop}
\label{prop3.4}
With probability at least $1-\delta$, for any $\theta \in \B{R}^d$, 
any $\lambda \in \B{R}_+$,  
\[
\frac{1}{n \lambda} \sum_{i=1}^n A_{\lambda,\beta}(\theta, X_i) 
- \frac{\beta \| 
\theta\|^2}{2n} - \frac{\log(\delta^{-1})}{n\lambda} 
\leq \mathds E\left( \langle \theta, X \rangle^2 \right) + \frac{\mathds E
\bigl( \|X\|^4 \bigr) }{\lambda \beta^2}. 
\]
Moreover with probability at least $1-\delta$, for any $\theta \in \B{R}^d$, 
any $\lambda \in \B{R}_+$,  
\begin{multline*}
\frac{1}{n \lambda} \sum_{i=1}^n A_{\lambda,\beta}(\theta, X_i) + \frac{\beta \| \theta\|^2}{2n} + \frac{\log(\delta^{-1})}{n\lambda} \\
\geq \mathds E\left( \langle \theta, X \rangle^2 \right) - \lambda \mathds E\left( \langle \theta, X \rangle^4 \right)
- \frac{6 \mathds E \left( \|X\|^2 \langle \theta, X \rangle^2\right)}{\beta}
 -  \frac{3 \mathds E \bigl( \lVert X \rVert^4 \bigr)}{\lambda \beta^2}. 
\end{multline*}
\end{prop}

\vskip2mm
\noindent
\begin{proof}
According to Proposition \vref{prop:2.1}, with probability at least 
$1 - \delta$, for any $\theta \in \B{R}^d$ and any $\lambda \in \B{R}_+$,  
\begin{multline*}
\frac{1}{n \lambda} \sum_{i=1}^{n} 
\Biggl[ \int \psi \bigl( \langle \theta', X_i 
\rangle^2 \bigr) \, \ud \rho_{\lambda^{1/2} \theta}(\theta') - 
\log \biggl( 1 + \frac{\lVert X_i \rVert^2}{\beta} \biggr) \Biggr] - 
\frac{\beta \lVert \theta \rVert^2}{2 n}  + \frac{\log(\delta^{-1})}{n \lambda} \\ 
\leq \frac{1}{\lambda} \int \log \Biggl\{ \B{E} \Biggl[ \exp \Biggl( 
\psi \bigl( \langle \theta', X \rangle^2 \bigr) 
- \log \biggl(1 + \frac{\lVert X \rVert^2}{\beta} \biggr) \Biggr) \Biggr] 
\Biggr\} 
\, \ud \rho_{\lambda^{1/2} \theta}(\theta'). 
\end{multline*} 
According to Lemma \vref{lem:5.1}, 
\begin{multline*}
\psi (t) - \log(1+u) \leq \log\left( \frac{1+t}{1+u}\right)  
\\ = \log\left( 1- u + \frac{t+u^2}{1+u}\right) \leq \log(1+t-u+u^2), \qquad t, u \in \B{R}_+.
\end{multline*}
Thus the right-hand side of the previous inequality is not greater than 
\[
\frac{1}{\lambda} \int \B{E} \bigl( \langle \theta', X \rangle^2 \bigr) \, \ud 
\rho_{\lambda^{1/2} \theta} (\theta') 
- \frac{\B{E} \bigl( \lVert X \rVert^2 \bigr)}{\lambda \beta } + 
\frac{\B{E} \bigl( \lVert X \rVert^4 \bigr)}{\lambda \beta^2 } =  \B{E} \bigl( \langle \theta, X \rangle^2 \bigr) + \frac{
\B{E} \bigl( \lVert X \rVert^4 \bigr)}{\lambda \beta^2 }.
\]
 In the same time, 
due to Lemma \ref{lem3.1}, its left-hand side is equal to 
\[
\frac{1}{n \lambda} \sum_{i=1}^n A_{\lambda, \beta}(\theta, X_i) 
- \frac{\beta \lVert \theta \rVert^2}{2 n} - \frac{\log(\delta^{-1})}{n \lambda
}.
\] 
This achieves the proof for the upper bound. 
Let us now come to the lower bound.\\
As a consequence of Lemma \vref{lem:5.1}, 
for any $t \in [0, 1]$ and any $y \in \B{R}_+$, 
\begin{align*}
- \psi(t) + \log (1 + y) & \leq \log \bigl( 1 - t 
 + t^2 \bigr) + \log( 1 + y )  \\ 
& = \log \Bigl( 1 - t + t^2 + \bigl( 1 - t + t^2 \bigr) y \Bigr)
\\ & \leq \log \bigl( 1 - t + t^2 + y \bigr).
\end{align*}
When $t \in [1, \infty[$ the same inequality is also obviously true:
\[ 
- \psi(t) + \log(1 + y) \leq \log( 1 + y) \leq \log (1 - t + t^2 + y). 
\] 
As a consequence, for any $x \in \B{R}^d$,  
\begin{multline*}
\int -\psi(\langle \theta', x \rangle^2 ) \, \mathrm d \rho_\theta(\theta') +  \log\left( 1+ \frac{\|x\|^2}{\beta}  \right) \\
\leq \int \log \left( 1-  \langle \theta', x \rangle^2+  \langle \theta', x \rangle^4 +\frac{\|x\|^2}{\beta}  \right) \, \mathrm d \rho_\theta(\theta')
\end{multline*}
Thus, according to the PAC-Bayesian inequality stated in Proposition 
\vref{prop:2.1}, with probability al least $1-\delta$, for any 
$\theta \in \B{R}^d$ and any $\lambda \in \B{R}_+$, 
\begin{multline*}
\frac{1}{n \lambda} \sum_{i=1}^n \left[\int- \psi(\langle \theta', X_i \rangle^2 ) \, \mathrm d \rho_{\lambda^{1/2} \theta}(\theta') + \log\left( 1+ \frac{\|X_i\|^2}{\beta}  \right) \right] 
-  \frac{\beta\|\theta\|^2}{2n  } -  \frac{\log(\delta^{-1})}{n \lambda}\\
\leq \frac{1}{\lambda} \int \log \Biggl\{ \B{E} \Biggl[ \exp \Biggl( 
- \psi \bigl( \langle \theta', X \rangle^2 \bigr) + \log \biggl( 
1 + \frac{\lVert X \rVert^2}{\beta} \biggr) \Biggr) \Biggr] \Biggr\} \, 
\ud \rho_{\lambda^{1/2} \theta} \bigl( \theta' \bigr)
\\ \leq \frac{1}{\lambda} \int \mathds E \left(- \langle \theta', X \rangle^2 +  
\langle \theta', X \rangle^4 +\frac{\| X \|^2}{\beta}   \right) \, \mathrm d \rho_{\lambda^{1/2} \theta} (\theta').
\end{multline*}
To conclude the proof, it is enough to use the explicit expression of 
the moments of a Gaussian random variable, remembering 
that, when $\theta'$ is distributed according to $\rho_{\lambda^{1/2} \theta}$, 
the distribution of $\langle \theta', X \rangle$
is equal to  $\C{N} 
\Bigl(\lambda^{1/2} \langle \theta, X \rangle, \lVert X \rVert^2/\beta \Bigr)$.  
\end{proof}

\vskip2mm
The next proposition defines an estimator of the quadratic 
form $\B{E} \bigl( \langle \theta, X \rangle^2 \bigr)$.
Note that, since we introduced a parameter $\lambda$ that takes care of the 
norm of $\theta$, we will assume in the following 
without loss of generality that $\theta \in  \mathds S_d$, 
the unit sphere of $\mathds R^d$.

\begin{prop}
\label{prop:3.6}
Let us assume that 
\[
\mathds E \bigl( \lVert X \rVert^4 \bigr) \leq T < \infty,
\]
for a known constant $T$.
For any $\theta \in \mathds S_d$, consider the estimator
of $\B{E} \bigl( \langle \theta, X \rangle^2 \bigr)$ defined as 
\[
\mathcal E(\theta) 
= \sup_{\lambda \in \B{R}^d} \left\{ \frac{1}{n \lambda} 
\sum_{i=1}^n A_{\lambda,\beta}(\theta, X_i)
 - \frac{\beta}{2n} - \frac{\log(\delta^{-1})}{n\lambda} - \frac{T}{\lambda \beta^2}
\right\}.
\]
With probability at least $1-\delta$, for any $\theta \in \mathds S_d$,
\[
\mathcal E(\theta) \leq \mathds E\left( \langle \theta, X \rangle^2 \right).
\]
Moreover, with probability at least $1-\delta$, for any $\theta \in \mathds S_d$,
\begin{multline*}
 \mathds E\left( \langle \theta, X \rangle^2 \right) 
 \leq \mathcal E(\theta) + 2\sqrt{ 2   \mathds E\left( \langle \theta, X \rangle^4 \right) \left( \frac{2T}{\beta^{2}} + \frac{ \log(\delta^{-1})}{n} \right)}\\
+  \frac{6 \mathds E \left( \|X\|^2 \langle \theta, X \rangle^2\right)}{\beta} + \frac{\beta }{n}. 
\end{multline*}
\end{prop}

\begin{rmk}
\label{rem:3.1}
Introducing $\ds \alpha = \frac{2 T n}{\beta^{2}}$, 
we can also express the previous 
bound as 
\begin{multline*}
\B{E} \bigl( \langle \theta, X \rangle^2 \bigr) \leq \C{E}(\theta) 
+ 2 \sqrt{ \frac{2}{n} \B{E} \bigl( \langle \theta, X \rangle^4 \bigr) 
\bigl[ \alpha + \log ( \delta^{-1} ) \bigr]} \\ 
\shoveright{+ 
3 \sqrt{\frac{2 \alpha}{Tn}} \B{E} \Bigl( \langle \theta, 
X \rangle^2 \lVert X \rVert^2 \Bigr) + 
\sqrt{ \frac{2 T}{ \alpha n }}} \\ 
\leq \C{E} ( \theta ) + \Biggl[ \, 2 \sqrt{2 \B{E} \bigl( \langle \theta, X
\rangle^4 \bigr)} + 3 \sqrt{\frac{2}{T}} \B{E} \Bigl( \langle \theta, X 
\rangle^2 \lVert X \rVert^2 \Bigr) \, \Biggr] \sqrt{\frac{\alpha}{n}} 
\\ \shoveright{ + \sqrt{\frac{2 T}{n \alpha}} + 2 \sqrt{\frac{2}{n} 
\B{E} \bigl( \langle 
\theta, X \rangle^4 \bigr) \log(\delta^{-1})}} \\ 
\leq \C{E}(\theta) + 5 \sqrt{ \frac{2 \alpha}{n} \B{E} \bigl( \langle \theta, X \rangle^4 \bigr)}
+ \sqrt{\frac{2 T}{n \alpha}} + 2 \sqrt{ \frac{2}{n} \B{E} 
\bigl( \langle \theta, X \rangle^4 \bigr) \log ( \delta^{-1})},
\end{multline*}
where the last inequality 
is a consequence of 
the Cauchy-Schwarz inequality
\[
\mathds E \left( \langle \theta, X \rangle^2 \|X\|^2 
\right) \leq \sqrt{ 
\mathds E \left( \langle \theta, X \rangle^4 \right) 
\mathds E(\|X\|^4) } \leq \sqrt{ T \ \mathds E\left( \langle \theta, X \rangle^4\right)}. 
\]
\end{rmk}
\begin{proof}
Proposition \ref{prop:3.6} follows from Proposition \ref{prop3.4} and the definition of the estimator $\mathcal E$. 
To get the second inequality, observe that the 
value of $\lambda$ minimizing 
\[
\lambda \mathds E\left( \langle \theta, X \rangle^4 \right) 
+ \frac{6 \mathds E \left( \|X\|^2 \langle \theta, X 
\rangle^2\right)}{\beta} + \frac{4 T }{\lambda 
\beta^2}+ \frac{\beta }{n} + \frac{2 \log(\delta^{-1})}{n\lambda}
\]
is given by 
\[
\lambda = \sqrt{ 2   \mathds E \bigl(  \langle \theta, X \rangle^4 \bigr)^{-1} 
\left( \frac{2T}{\beta^2} + \frac{ \log(\delta^{-1})}{n} \right)}. 
\]
\end{proof}
In the following proposition, 
we make the estimator 
adaptive in $\alpha$ 
as well as in $\lambda$ 
and we introduce our estimator $\widehat G$ of the Gram matrix $G$. 
\begin{prop}
\label{prop:6.5}
Let us assume that $\B{E} \bigl( \lVert X \rVert^4 \bigr) \leq T < 
\infty$, where $T$ is a known constant. Consider the estimator 
\begin{multline*}
\wt{\C{E}} ( \theta ) = \sup_{\lambda \in \B{R}_+} \sup_{k \in \B{N}} 
\frac{1}{n \lambda} \sum_{i=1}^n \Biggl[ \varphi_2 \Biggl( 
\sqrt{\lambda} \langle \theta, X_i \rangle, 
\Biggl( \frac{\exp(k)}{10 T n} \Biggr)^{1/4} \lVert X_i \rVert \, \Biggr) 
\\ - \log \Biggl( 1 +  \sqrt{ \frac{\exp(k)}{ 10 T n}} \lVert 
X_i \rVert^2 \, \Biggr) \Biggr]
- \frac{1}{n} \sqrt{\frac{5 T}{2 \exp(k)}} - \frac{\exp(k)}{10 n \lambda} \\ - 
\frac{\log \bigl[ (k+1)(k+2) / \delta \bigr]}{ n \lambda}, \qquad \theta \in 
\B{S}_d.
\end{multline*}
With probability at least $1 - \delta$, for any $\theta \in \B{S}_d$, 
\[ 
\wt{\C{E}} ( \theta ) \leq \B{E} \bigl( \langle \theta, X \rangle^2 \bigr).
\] 
With probability at least $1 - \delta$, for any $\theta \in \B{S}_d$, 
\[
\B{E} \bigl( \langle \theta, X \rangle^2 \bigr) \leq \wt{\C{E}} ( \theta )
+ B(\theta), 
\]
where
\begin{multline*}
B(\theta) =  
2 \sqrt{\frac{\B{E} \bigl( \langle \theta, X \rangle^4 \bigr)}{n}} 
\; \left\{\rule{0pt}{25pt}\right.
3.3 \biggl( \frac{T}{\B{E}\bigl( \langle \theta, X \rangle^4 \bigr)} \biggr)^{1/4} 
\\ 
+ \sqrt{ 4 \log \Biggl( \, \frac{1}{2} \log \biggl( \frac{T}{\B{E} 
\bigl( \langle \theta, X \rangle^4 \bigr)} \biggr) + \frac{5}{2} 
\; \Biggr) + 2 \log (\delta^{-1})} 
\; \left.\rule{0pt}{25pt}\right\}.
\end{multline*}
Consider an estimator $\wh{G} \in \B{R}^{d \times d}$ such that 
\[ 
0 \leq \inf_{\theta \in \B{S}_d} \langle \theta, \wh{G} \, \theta \rangle
- \wt{\C{E}}(\theta) 
\] 
and 
\begin{multline*}
\sup_{\theta \in \B{S}_d} \langle \theta , \wh{G} \, \theta 
\rangle - \wt{\C{E}}(\theta) = \inf \biggl\{ 
\; \sup_{\theta \in \B{S}_d} 
\langle \theta, M \, \theta \rangle - \wt{\C{E}}(\theta) \, : \, 
M \in \B{R}^{d \times d}, \\ M = M^{\top}, \; 0 \leq \inf_{\theta \in \B{S}_d}
\langle \theta, M \, \theta \rangle - \wt{\C{E}}(\theta) \biggr\}.
\end{multline*}
With probability at least $1 - 2 \delta$, 
\[ 
\bigl\lVert \wh{G} - G \bigr\rVert_{\infty} 
\leq \sup_{\theta \in \B{S}_d} B(\theta).  
\] 
\end{prop}
\begin{rmk}
It is interesting to rephrase this result in terms of the directional 
kurtosis 
\[ 
\kappa(\theta) = 
\begin{cases} 
\ds \frac{\B{E} \bigl( \langle \theta, X \rangle^4 
\bigr)}{ \B{E} \bigl( \langle \theta, X \rangle^2 \bigr)^2}, & 
\B{E} \bigl( \langle \theta, X \rangle^2 \bigr) > 0, \\ 
1, & \text{ otherwise.}
\end{cases} 
\] 
We obtain with probability at least $1 - 2 \delta$, 
\begin{multline*}
1 - 2 \sqrt{\frac{\kappa(\theta)}{n}} 
\left\{\rule{0pt}{25pt}\right.
3.3 \Biggl( \, \frac{T}{\kappa(\theta) \B{E}\bigl( \langle \theta, X \rangle^2 \bigr)^2} \, \Biggr)^{1/4} 
\\ 
+ \sqrt{ 4 \log \Biggl( \, \frac{1}{2} \log \Biggl( \frac{T}{
\kappa(\theta) \B{E} 
\bigl( \langle \theta, X \rangle^2 \bigr)^2} \biggr) + \frac{5}{2} 
\biggr) + 2 \log (\delta^{-1})} 
\; \left.\rule{0pt}{25pt}\right\}
\leq \frac{\wt{\C{E}}(\theta)}{\B{E} \bigl( \langle \theta, X \rangle^2 \bigr)
} \leq 1, 
\end{multline*}
with the appropriate convention that $r / 0  = + \infty$ when $r > 0$ and 
$0/0 = 1$. 
This inequality shows under which circumstances it is possible 
to estimate the order of magnitude of 
$\B{E}(\langle \theta, X \rangle^2 \bigr)$ and consequently 
the eigenvalues of the Gram matrix $G$. Indeed, introducing 
$\kappa_* = \sup_{\theta \in \B{S}_d} \kappa(\theta)$, we deduce 
with probability at least $1 - 2 \delta$ a bound 
of the form 
\[ 
1 - \sqrt{\frac{f \Bigl( \kappa_*, 
\B{E} \bigl( \langle \theta, X \rangle^2 \bigr)
 \Bigr)}{n}} \leq \frac{\wt{\C{E}}(\theta)}{\B{E} \bigl( \langle \theta, 
X \rangle^2 \bigr)} \leq 1, 
\]  
where the function $\sigma \mapsto F(\kappa_*, \sigma) = \sigma \Bigl( 1 - 
\sqrt{f(\kappa_*, 
\sigma)/n} \Bigr)$ is non-decreasing. 
Let us write $G = \B{E} \bigl( X X^{\top} \bigr)$ as 
\[ 
G = \sum_{i=1}^d \sigma_i  e_i e_i^{\top},
\] 
where 
$(e_1, \dots, e_d)$ is an orthonormal basis of eigenvectors 
and where
$\sigma_1 \geq \sigma_2 \geq \cdots \geq \sigma_d$ are 
the eigenvalues of $G$ counted with their multiplicities and sorted in 
decreasing order.
Introducing $\C{L}_i$, the set of all linear subspaces of $\B{R}^d$ 
of dimension $i$, it is well known that 
\[ 
\sigma_i = \sup \Bigl\{ \inf \Bigl\{ \langle \theta, G \theta \rangle, 
\; \theta \in L \cap \B{S}_d \Bigr\}, \; L \in \C{L}_i \Bigr\}.
\] 
A proof can for instance be found in \cite[page 62]{Kato2}. 
Based on this formula, we can introduce the estimator 
\[ 
\wh{\sigma}_i = 
\sup \Bigl\{ \inf \Bigl\{ \wt{\C{E}}(\theta), 
\; \theta \in L \cap \B{S}_d \Bigr\}, \; L \in \C{L}_i \Bigr\}.
\] 
It is such that 
\begin{multline*}
F \bigl(\kappa_*, \sigma_i \bigr) = 
F \Bigl( \kappa_*, 
\sup \Bigl\{ \inf \Bigl\{ \langle \theta, G \theta 
\rangle, \; \theta \in L \cap \B{S}_d \Bigr\}, \; L 
\in \C{L}_i \Bigr\} \Bigr) \\ =   
\sup \Bigl\{ \inf \Bigl\{ F \bigl( \kappa_*, \langle \theta, G \theta 
\rangle \bigr), \; \theta \in L \cap \B{S}_d \Bigr\}, \; L 
\in \C{L}_i \Bigr\}
\\ \leq \wh{\sigma}_i \leq 
\sup \Bigl\{ \inf \Bigl\{ \langle \theta, G \theta 
\rangle, \; \theta \in L \cap \B{S}_d \Bigr\}, \; L 
\in \C{L}_i \Bigr\} = \sigma_i,
\end{multline*}
proving that with probability at least $1 - 2 \delta$, 
\[ 
1 - \sqrt{\frac{f(\kappa_*, \sigma_i)}{n}} \leq \frac{\wh{\sigma}_i}{\sigma_i} 
\leq 1, \qquad 1 \leq i \leq d. 
\] 
\end{rmk}
\begin{proof}[Proof of Proposition \vref{prop:6.5}] 
The optimal value of $\alpha$ in the last bound given in Remark 
\vref{rem:3.1} is given by 
\[
\alpha_* = \frac{1}{5} \sqrt{\frac{T}{\B{E} \bigl( \langle \theta, X \rangle^4
\bigr)}} \geq \frac{1}{5} \sqrt{ \frac{\B{E} \bigl( \lVert X \rVert^4 \bigr)}{
\B{E} \bigl( \langle \theta, X \rangle^4 \bigr)}} \geq \frac{1}{5}. 
\]
According to the simplified inequality stated at the end of Remark 
\ref{rem:3.1}, with probability at least $1 - \delta$, for any 
$\theta \in \B{S}_d$, 
\begin{multline*}
\B{E} \bigl( \langle \theta, X \rangle^2 \bigr) 
\leq 
\C{E}(\theta) + 2 \sqrt{\frac{10}{n}} \Bigl[ T \B{E} \bigl( \langle \theta, 
X \rangle^4 \bigr) \Bigr]^{1/4} \frac{\bigl( \sqrt{ \alpha/\alpha_*} + \sqrt{ 
\alpha_* / \alpha}\bigr)}{2}\\
\shoveright{+ 2 \sqrt{ \frac{2}{n} \B{E} \bigl( \langle \theta, X \rangle^4
\bigr) \log(\delta^{-1})
}}\\
= 
\C{E}(\theta) + 2 \sqrt{\frac{10}{n}} \Bigl[ T \B{E} \bigl( \langle \theta, 
X \rangle^4 \bigr) \Bigr]^{1/4} \cosh \biggl( \frac{1}{2} 
\log(\alpha / \alpha_*) \biggr) 
\\ + 2 \sqrt{ \frac{2}{n} \B{E} \bigl( \langle \theta, X \rangle^4
\bigr) \log(\delta^{-1})}.
\end{multline*}
We will take a weighted union bound on all values of $\alpha$ 
belonging to \linebreak $\Bigl\{ \exp(k) / 5 \, : \, k \in \B{N} \Bigr\}$. To perform this, 
we have to modify accordingly the definition of the estimator
and consider the estimator $\wt {\mathcal E}$
defined in the proposition. 
In this change of definition, we have replaced $
\beta$ with $\ds \sqrt{\frac{10 \, T n}{\exp(k)}}$ and $\delta$ with 
$\ds \frac{\delta}{(k+1)(k+2)}$, and we have taken the supremum 
in $k \in \B{N}$ as well as in $\lambda \in \B{R}_+$. 
As 
\[\sum_{k \in \B{N}} \frac{\delta}{(k+1)(k+2)} = \delta
\] 
we get from Proposition \vref{prop3.4} that with probability at least $1 - 
\delta$, for any $\theta \in \B{S}_d$, $\wt{\C{E}} ( \theta ) \leq 
\B{E} \bigl( \langle \theta, X \rangle^2 \bigr)$. 
Recalling that $\ds \alpha = \frac{2Tn}{\beta^2} = \frac{\exp(k)}{5}$, 
we get 
with probability 
at least $1 - \delta$, for any $\theta \in \B{S}_d$, 
\begin{multline*}
\B{E} \bigl( \langle \theta, X \rangle^2 \bigr) \leq \wt{\C{E}} ( \theta )
\\ + \inf_{k \in \B{N}} 2 \sqrt{\frac{10}{n}} \Bigl[ T 
\B{E} \bigl( \langle \theta, X \rangle^4 \bigr) \Bigr]^{1/4} 
\cosh \Biggl( \frac{1}{2} \log \biggl( \frac{\exp(k) }{5 \alpha_*} 
\biggr) \Biggr)
\\ + 2 \sqrt{\frac{2}{n} \B{E} \bigl( \langle \theta, X \rangle^4 
\bigr) \log \bigl[ (k+2)^2 / \delta \bigr]}.
\end{multline*}
(We can take the infimum in $k$ because the inequality holds 
with probability $1 - \delta$ for any value of $k \in \B{N}$). 
We can now choose $k$ to be the closest integer to $\log(5 \alpha_*)$
(that is known to be a non-negative quantity). It is 
such that $\ds \left\lvert \log \biggl( \frac{\exp(k)}{5 \alpha_*} 
\biggr) \right\rvert \leq \frac{1}{2}$ and therefore 
\[ 
k + 2 \leq \log \bigl( 5 \alpha_* \bigr) + \frac{5}{2} = 
\frac{1}{2} \log \biggl( \frac{T}{\B{E} \bigl( \langle 
\theta, X \rangle^4 \bigr)} \biggr) + \frac{5}{2}.
\]  
Remarking that $\sqrt{10} \cosh(1/4) \leq 3.3$ ends the proof. 
\end{proof}

\section{Linear least squares regression}
\label{sec:7}

Consider a couple of random variables $(X, Y) \in \B{R}^{d} \times \B{R}$
whose distribution is assumed to be unknown. Let 
\[ 
(X_1, Y_1), \dots, (X_n, Y_n)
\]  
be an observed sample made of $n$ independent copies of $(X, Y)$. 
In this section, we consider the question of estimating 
\[ 
\inf_{\theta} \B{E} \bigl[ \bigl( \langle \theta, X \rangle - Y 
\bigr)^2 \bigr]. 
\] 
Introduce the Gram matrix 
\[ 
G = \B{E} \bigl( X X^{\top} \bigr) \in \B{R}^{d \times d}, 
\]
the vector 
\[
V = \B{E} \bigl( Y X \bigr) \in \B{R}^d,
\]
and the risk function 
\[
R(\theta) = \bigl\langle \theta, G \theta \bigr\rangle - 2 
\bigl\langle \theta, V \bigr\rangle. 
\]
Remark that 
\[ 
\B{E} \bigl[ \bigl( \langle \theta, X \rangle - Y \bigr)^2 \bigr]
= \B{E} (Y^2) + R(\theta), \qquad \theta \in \B{R}^d, 
\] 
so that minimizing the quadratic loss is equivalent to minimizing $R$. 

We have seen in the previous sections various methods to estimate 
$G$ and $V$. 
As a straightforward consequence, we state a first result, 
concerning the minimization over a bounded domain. 
\begin{prop}
\label{prop:8.1}
Assume that $\wh{G} \in \B{R}^{d \times d}$ and $\wh{V} 
\in \B{R}^d$ are such that 
\begin{equation}
\label{eq:8.3.2}
\lVert \wh{G} - G \rVert_{\infty} \leq \epsilon, \text{ and } 
\lVert \wh{V} - V \rVert \leq \eta.
\end{equation}
Assume also that $\wh{G}$ is a 
symmetric positive semi-definite matrix. 
Let $\Theta$ be a closed bounded set in $\B{R}^d$ and 
let $\ds B = \sup_{\theta \in \Theta} \lVert \theta \rVert$.
Consider the estimated risk 
\[ 
\wh{R}(\theta) = \bigl\langle \theta, \wh{G} \, \theta \bigr\rangle 
- 2 \, \bigl\langle \theta, \wh{V} \bigr\rangle, \qquad \theta \in \B{R}^d,
\] 
and an estimator $\ds \wh{\theta} \in \arg \min_{\Theta} \wh{R}$. 
It is such that
\[
R(\wh{\theta}) - \inf_{\Theta} R \leq 2 B \bigl( \epsilon B + 2 \eta 
\bigr). 
\] 
\end{prop}
\begin{proof}
Remark that 
\begin{align*}
R(\wh{\theta})  & \leq \wh{R}(\wh{\theta}) + B^2 \epsilon + 2 B \eta
= \inf_{\theta \in \Theta} \wh{R}(\theta) + B^2 \epsilon + 2 B \eta
\\ & \leq \inf_{\theta \in \Theta} R(\theta) + 2 B^2 \epsilon + 4 B \eta.
\end{align*}
\end{proof}
\begin{cor}
Assume that we know constants $v, T, v', T'$ such that 
\begin{align*}
& \sup_{\theta \in \B{S}_d} \B{E} \bigl( \langle \theta, X \rangle^4 
\bigr) \leq v < \infty,\\  
& \B{E} \bigl( \lVert X \rVert^4 \bigr) \leq T < \infty, \\ 
& \sup_{\theta \in \B{S}_d} \B{E} \bigl( Y^2 \langle \theta, X
\rangle^2 \bigr) \leq v' < \infty, \\ 
& \B{E} \bigl( Y^2 \lVert X \rVert^2  \bigr) \leq T' < \infty,
\end{align*}
Using Propositions \vref{prop:2.3} and \vref{prop:3.2}, we can define 
estimators $\wh{G}$ and $\wh{V}$ such that with probability at least 
$1 - 2 \delta$, 
\begin{align*}
& \lVert \wh{G} - G \rVert_{\infty} \leq \epsilon = 2 \sqrt{ \frac{2 v}{n} 
\bigl( 2 \log(\delta^{-1} + 12 \sqrt{T/v} \bigr)}\\  
\text{ and } \qquad 
& \lVert \wh{V} - V \rVert 
\leq \eta = 2 \Bigl( \sqrt{T'/n} + \sqrt{2 v' \log(\delta^{-1})/n} \Bigr).
\end{align*}
Consequently, the estimator $\wh{\theta}$ of the previous proposition 
based on $\wh{G}$ and $\wh{V}$ 
is such that with probability at least $1 - 2 \delta$, 
\[ 
R(\wh{\theta}) - \inf_{\Theta} R \leq \C{O} \Biggl( \sqrt{
\frac{\log(\delta^{-1})}{n}} \; \Biggr), 
\] 
where the constant hiding behind the notation $\C{O}$ 
depends only on $v, T, v', T'$ and $\ds \sup_{\theta \in \Theta} \lVert 
\theta \rVert$. 
\end{cor}
\begin{rmk} 
We get only a slow speed of order $n^{-1/2}$ and not $n^{-1}$, but 
we think it is the price to pay to have a dimension-free bound 
under such hypotheses.
\end{rmk}

In the following, we will release the constraint that $\theta$ 
belongs to a bounded domain. We will also propose conditions 
under which a fast rate of order \linebreak $ \C{O} \bigl( \log(\delta^{-1})/n
\bigr)$ is possible. We will be interested first in defining 
some non-asympto\-tic confidence region for $\theta_* \in \arg \min_{
\theta \in \B{R}^d} R(\theta) $. 
We will broaden our analysis to the estimation of the ridge 
regression $\ds \theta_{\lambda} \in \arg \min_{\theta \in \B{R}^d} 
\bigl( R(\theta) + \lambda \lVert \theta \rVert^2 \bigr)$, since this extension 
is quite natural in this context. Indeed, the ridge regression 
problem consists in minimizing $R$ on a ball centered at the origin, 
and ridge regressors, as we will see, will anyhow play a role 
in the definition of a robust estimator. 

\begin{prop}
Make the same assumptions as at the beginning of Proposition 
\vref{prop:8.1} and consider some parameter $\lambda \in \B{R}_+$. 
Introduce the ridge regression loss function
\[ 
R_{\lambda} (\theta) = R(\theta) + \lambda \lVert \theta \rVert^2 = 
\langle \theta, (G + \lambda I) \theta \rangle - 2 \langle \theta, V \rangle 
\] 
and its empirical counterpart
\[ 
\wh{R}_{\lambda}(\theta) = \wh{R}(\theta) + \lambda \lVert \theta \rVert^2
= \langle \theta, ( \wh{G} + \lambda I) \theta \rangle - 2 \langle \theta, 
\wh{V} \rangle.
\]  
Let $\ds \theta_{\lambda} \in \arg \min_{\theta \in \B{R}^d} R_{\lambda}$ and 
$\ds \wh{\theta}_{\lambda} \in \arg \min_{\theta \in \B{R}^d} 
 \wh{R}_{\lambda}(\theta) $. 
Define the confidence region
\[ 
\wh{\Theta}_{\lambda} = \Bigl\{ \theta \in \B{R}^d \, : \, 
\bigl\lVert (\wh{G} + \lambda) \bigl( \theta - \wh{\theta}_{\lambda} \bigr) 
\bigr\rVert \leq  \lVert \theta \rVert \epsilon + \eta \Bigr\}.
\] 
On the event defined 
by equation \myeq{eq:8.3.2}, 
\[ 
\theta_{\lambda} \in \wh{\Theta}_{\lambda}.
\] 
Moreover, for any estimator $\wh{\theta} \not\in \wh{\Theta}_{\lambda}$, 
the improved pick 
\[ 
\wt{\theta}  \in \arg \min_{\theta \in \B{R}^d} \left\{ \, \wh{R}_{\lambda} ( \theta ) 
- \wh{R}_{\lambda} ( \wh{\theta} ) 
+ \epsilon \bigl\lVert \theta - \wh{\theta} \bigr\rVert^2  
+ 2 \lVert \theta - \wh{\theta} \rVert \Bigl( 
\epsilon \lVert \wh{\theta} \bigr\rVert + \eta \Bigr) \; \right\} 
\] 
is such that 
\[ 
R_{\lambda} (\wt{\theta}) < R_{\lambda} (\wh{\theta}),
\] 
and more precisely such that 
\[ 
R_{\lambda}(\wt{\theta}) - R_{\lambda}(\wh{\theta}) \leq \wh{R}_{\lambda}(
\wt{\theta}) - \wh{R}_{\lambda}( \wh{\theta} ) + \bigl\lVert 
\wt{\theta} - \wh{\theta} \bigr\rVert \Bigl( \epsilon \bigl\lVert 
\wt{\theta} + \wh{\theta} \bigr\rVert + 2 \eta \Bigr) < 0. 
\] 
\end{prop}
\begin{proof}
Note that for any $\theta, \xi \in \B{R}^d$, 
\begin{multline*}
R_{\lambda}(\xi) - R_{\lambda}(\theta)  
= \langle \theta - \xi, (G + \lambda I) (\theta + \xi) \rangle 
- 2 \langle \theta - \xi, V \rangle \\ \leq 
\wh{R}_{\lambda}(\xi) - \wh{R}_{\lambda}(\theta) + 
\lVert \xi - \theta \rVert \bigl( \epsilon \lVert \xi + \theta \rVert + 2 \eta
\bigr) \\ \leq  
\wh{R}_{\lambda}(\xi) - \wh{R}_{\lambda}(\theta) + 
\epsilon \lVert \xi - \theta \rVert^2 + 
2 \lVert \xi - \theta \rVert \bigl( \epsilon \lVert \theta \rVert 
+ \eta \bigr) \overset{\text{def}}{=} \gamma_{\lambda}(\theta, \xi).
\end{multline*}
As $\xi \mapsto \gamma_{\lambda}(\theta, \xi)$ is strictly convex, 
$\ds \inf_{\xi \in \B{R}^d} \gamma_{\lambda}(\theta, \xi) = 0 = 
\gamma_{\lambda}(\theta, \theta)$ 
if and only if its subdifferential satisfies 
\[ 
0 \in \frac{\partial}{\partial \xi}_{| \xi = \theta} \gamma_{\lambda}(
\theta, \xi) = 2 ( \wh{G} + \lambda I) \theta - 2 \wh{V} + 
2 \B{B}_d \bigl( \epsilon \lVert \theta \rVert + \eta \bigr), 
\]  
where $\B{B}_d$ is the unit ball of $\B{R}^d$. Remarking that 
$\wh{V} = (\wh{G} + \lambda I) \wh{\theta}_{\lambda}$, we see 
that this is equivalent to 
\[ 
\lVert ( \wh{G} + \lambda I) ( \theta - \wh{\theta}_{\lambda}) \rVert 
\leq  \epsilon \lVert \theta \rVert + \eta. 
\] 
To complete the proof, it is enough to remark that, due to its definition, 
\[
0 \geq \inf_{\xi \in \B{R}^d} \gamma_{\lambda}(\theta_{\lambda}, \xi) \geq 
\inf_{\xi \in \B{R}^d} R_{\lambda}(\xi) - R_{\lambda}(\theta_{\lambda}) = 0,
\]
so that $\theta_{\lambda} \in \wh{\Theta}_{\lambda}$. 
\end{proof}
Note that $\wt{\theta} $ 
is the solution of a strictly convex minimization problem. 
It is characterized by the equation  
\[
\bigl( \wh{G} + \lambda I \bigr) \theta  - \wh{V} + \epsilon \bigl( \theta 
- \wh{\theta} \bigr) + \frac{ \theta - \wh{\theta}}{\lVert \theta 
- \wh{\theta} \rVert} \bigl( \epsilon \lVert \wh{\theta} \rVert + \eta 
\bigr) = 0.
\]
In view of the shape of the confidence region, it is natural 
to consider the estimator
\[
\wt{\theta}_{\lambda} \in \arg \min_{\theta \in \wh{\Theta}_{\lambda}} 
\lVert \theta \rVert. 
\] 
\begin{prop}
Let $\xi \in \wh{\Theta}_{\lambda}$ be any parameter value within 
the above defined confidence region. Under the event defined by equation 
\myeq{eq:8.3.2}, it is such that 
\[ 
\bigl\lVert (G + \lambda I ) \bigl( \xi - \theta_{\lambda} \bigr) 
\bigr\rVert \leq 2 \bigl( \epsilon \lVert \xi \rVert + \eta \bigr).
\]  
In particular, since $\theta_{\lambda} \in \wh{\Theta}_{\lambda}$, 
we see from the definition of $\wt{\theta}_{\lambda}$ that 
$\lVert \wt{\theta}_{\lambda} \rVert \leq \lVert \theta_{\lambda} 
\rVert$ and therefore that
\[
\bigl\lVert (G + \lambda I ) \bigl( \wt{\theta}_{\lambda} - \theta_{\lambda} 
\bigr) \bigr\rVert^2 \leq 4 \bigl( \epsilon \lVert \theta_{\lambda} \rVert + 
\eta \bigr)^2. 
\] 
Thus, when $\epsilon = \C{O} \bigl( \sqrt{ \log(\delta^{-1}) / n} \bigr)$ 
and $\eta = \C{O} \bigl( \sqrt{\log(\delta^{-1}) / n} \bigr)$, we get a 
convergence speed
of order $\C{O} \bigl( \log(\delta^{-1}) / n \bigr)$, 
but for a modified definition of the loss function.  
Using a basis $\bigl( e_i, 1 \leq i \leq d \bigr)$ of eigenvectors of 
$G$, with corresponding eigenvalues $\sigma_1 \geq \sigma_2 \geq \cdots 
\geq \sigma_d \geq 0$, 
we see more precisely that for any $\theta \in \B{R}^d$,  
\[ 
R_{\lambda}(\theta) - R_{\lambda}(\theta_{\lambda}) 
= \sum_{i=1}^d \bigl( \sigma_i + \lambda \bigr) 
\langle \theta - \theta_{\lambda}, e_i \rangle^2, 
\] 
whereas
\[
\bigl\lVert (G + \lambda I) \bigl( \theta - \theta_{\lambda} \bigr) \bigr\rVert^2 
= \sum_{i=1}^d (\sigma_i + \lambda)^2 
\langle \theta - \theta_{\lambda}, e_i \rangle^2 = \frac{1}{4} \Bigl\lVert 
\nabla R_{\lambda}(\theta) \Bigr\rVert^2. 
\] 
The relation between the two risks is that 
\begin{align*}
(\sigma_d + \lambda) \bigl[ R_{\lambda}(\theta) - R_{\lambda} ( 
\theta_{\lambda} ) \bigr] & \leq 
\bigl\lVert (G + \lambda I ) ( \theta - \theta_{\lambda} ) 
\bigr\rVert^2 \\ & \leq ( \sigma_1 + \lambda) \bigl[ R_{\lambda}(\theta) - 
R_{\lambda} ( \theta_{\lambda} ) \bigr]. 
\end{align*}
Consequently
\[ 
R_{\lambda}(\wt{\theta}_{\lambda}) - R_{\lambda}(\theta_{\lambda}) 
\leq \frac{4}{\sigma_d + \lambda} \bigl( \epsilon 
\lVert \wt{\theta}_{\lambda} \rVert + \eta \bigr)^2 
\leq \frac{4}{\sigma_d + \lambda} \bigl( \epsilon \lVert \theta_{\lambda}
\rVert + \eta \bigr)^2. 
\] 
\end{prop}
\begin{proof}
For any $\xi \in \wh{\Theta}_{\lambda}$, 
\[ 
\lVert (G + \lambda I)(\xi - \theta_{\lambda}) \rVert = \lVert (G+\lambda I) 
\xi - V \rVert \leq \lVert (\wh{G} + \lambda I) \xi - \wh{V} 
\rVert + \epsilon \lVert \xi \rVert + \eta 
\leq 2 \bigl( \epsilon \lVert \xi \rVert + \eta \bigr),
\] 
from which the other statements made in the proposition 
are straightforward consequences.
\end{proof}

From this proposition, we conclude that we have a 
dimension-free bound for \linebreak  $\bigl\lVert (G + \lambda I) 
\bigl( \wt{\theta}_{\lambda} - \theta_{\lambda} \bigr) \bigr\rVert^2$, 
whereas the bound we obtain for $R_{\lambda}(\wt{\theta}_{\lambda}) -  
R_{\lambda}(\theta_{\lambda})$ depends on the dimension 
through $\sigma_d + \lambda$, so that it is dimension-free 
only for large enough values of $\lambda$.

For small values of $\lambda$ depending on $n$, we can obtain
a dimension-free slow rate in the following way. 
Remark that, 
since $\ds \sigma_i \leq \frac{(\sigma_i + \lambda)^2}{4 \lambda}$,  
\begin{multline*}
R_0(\wt{\theta}_{\lambda}) - R_0(\theta_0) = \sum_{i=1}^d \sigma_i 
\langle \wt{\theta}_{\lambda} - \theta_0 , e_i \rangle^2
\leq \sum_{i=1}^d \frac{(\sigma_i + \lambda)^2}{4 \lambda} 
\langle \wt{\theta}_{\lambda} - \theta_0, e_i \rangle^2
\\ = \frac{1}{4 \lambda} \bigl\lVert (G + \lambda I) 
(\wt{\theta}_{\lambda} - \theta_0 ) \bigr\rVert^2.
\end{multline*}
Since $V = G \theta_0 = (G + \lambda I) \theta_{\lambda}$,  
\begin{multline*}
\bigl\lVert (G + \lambda I) ( \wt{\theta}_{\lambda} - \theta_0 ) \bigr\rVert = \bigl\lVert (G + \lambda I) ( \wt{\theta}_{\lambda} - \theta_{\lambda} ) - 
\lambda \theta_0 \bigr\rVert
\\ \leq \bigl\lVert (G + \lambda I) ( \wt{\theta}_{\lambda} - \theta_{\lambda} 
) \bigr\rVert + \lambda \lVert \theta_0 \rVert
\leq 2 \bigl( \epsilon \lVert \theta_{\lambda} \rVert + \eta \bigr) + \lambda 
\lVert \theta_0 \rVert. 
\end{multline*}
Moreover, $\lVert \theta_{\lambda} \rVert \leq \lVert \theta_0 \rVert$, 
indeed, 
\[ 
R_{\lambda}(\theta_{\lambda}) = R_0(\theta_{\lambda}) + \lambda \lVert \theta_{\lambda} \rVert^2
\leq R_0(\theta_0) + \lambda \lVert \theta_0 \rVert^2 \leq 
R_0(\theta_{\lambda}) + \lambda \lVert \theta_0 \rVert^2.  
\] 
Therefore, 
\[ 
\lVert (G + \lambda I)(\wt{\theta}_{\lambda} - \theta_0) \rVert 
\leq 2 \bigl[ (\epsilon + \lambda/2)  \lVert \theta_0 \rVert + \eta \bigr] 
\] 
and coming back to $R_0$, 
\[ 
R_0(\wt{\theta}_{\lambda}) - R_0(\theta_0) \leq 
\frac{1}{\lambda} \bigl[ (\epsilon + \lambda / 2  ) \lVert 
\theta_0 \rVert + \eta \bigr]^2.
\] 
Choose 
$\lambda = 2 (\epsilon + \eta)$ to obtain
\[ 
R_0(\wt{\theta}_{2(\epsilon + \eta)}) - R_0(\theta_0) \leq 
\bigl[ \lVert \theta_0 \rVert + 1/2 \bigr] 
\bigl[ (2 \epsilon + \eta) \lVert \theta_0 \rVert + \eta \bigr].  
\]  
This is a dimension-free bound for $R_0(\wt{\theta}_{\lambda}) - 
R_0(\theta_0)$, 
but it is of order \linebreak $\C{O}\bigl( \sqrt{\log(\delta^{-1})/ n} \bigr)$ instead 
of $\C{O} \bigl( \log(\delta^{-1}) / n \bigr)$.
Notice that it is adaptive in $\lVert \theta_0 \rVert$, though.  

To get faster dimension-free rates for $R_0(\theta)$, we need 
to introduce some restrictions. 

First of all, let us notice that the previous results hold 
uniformly in any linear subspace of $\B{R}^d$. 

\begin{prop}
Let us make the same assumptions as in Proposition \vref{prop:8.1}.
For any linear subspace $L$ of $\B{R}^d$, define
\begin{align*}
\theta_{L,\lambda} & \in \arg \min_{\xi \in L} R_{\lambda}(\xi), \\
\wh{\theta}_{L, \lambda} & \in \arg \min_{\xi \in L} \wh{R}_{\lambda}(\xi).
\end{align*}
Let
\[
\pi_L \theta  = \arg \min_{\xi \in L}  \lVert \xi - \theta \rVert
\] 
be the orthogonal projection on L and let
\begin{align*}
\wh{\Theta}_{L, \lambda} & = \Bigl\{ \xi \in L : \bigl\lVert 
\pi_L (\wh{G} + \lambda I) (\xi - \wh{\theta}_{L,\lambda}) \bigr\rVert 
\leq \epsilon \lVert \xi \rVert + \eta \Bigr\} \\ 
\text{and} \quad \wt{\theta}_{L, \lambda} & \in \arg \min_{\xi \in \wh{\Theta}_{L, \lambda}} 
\lVert \xi \rVert.
\end{align*}
Finally introduce the least eigenvalue of $\pi_L G \pi_L$
\[
\sigma_L  = \inf \Bigl\{ \, \lVert G \xi \rVert \, : \, \xi \in L, 
\, \lVert \xi \rVert 
= 1 \, \Bigr\}. 
\]
Whenever equation \myeq{eq:8.3.2} is satisfied, for any linear subspace $L$ 
of $\B{R}^d$ and any parameter $\lambda \in \B{R}_+$,  
\begin{align*}
\bigl\lVert \pi_{L} (G + \lambda I) ( \wt{\theta}_{L, \lambda} - 
\theta_{L, \lambda} ) \bigr\rVert^2 & \leq 
4 \bigl( \epsilon \lVert \wt{\theta}_{L, \lambda} \rVert + \eta \bigr)^2 
\\ & \leq 
4 \bigl( \epsilon \lVert \theta_{L, 
\lambda} \rVert + \eta \bigr)^2, \\  
\text{and } R_{\lambda}(\wt{\theta}_{L, \lambda}) - R_{\lambda} 
(\theta_{L, \lambda}) & \leq 
\frac{4}{\sigma_L + \lambda} \bigl( \epsilon 
\lVert \wt{\theta}_{L, \lambda} \rVert + \eta \bigr)^2 
\\ & \leq \frac{4}{\sigma_L + \lambda} \bigl( 
\epsilon \lVert \theta_{L, \lambda} \rVert + \eta \bigr)^2. 
\end{align*}
\end{prop}
Remark that we can estimate $\sigma_L$ by 
\[ 
\wh{\sigma}_L = \inf \Bigl\{ \; \lVert \wh{G} \xi \rVert \, : \, 
\xi \in L, \; \lVert \xi \rVert = 1 \; \Bigr\}.
\] 
It is such that, for any linear subspace $L$, 
\[ 
\wh{\sigma}_L - \epsilon \leq \sigma_L \leq \wh{\sigma}_L + \epsilon.
\] 

Obtaining a fast convergence rate for the minimization of $R_{\lambda}(\theta)$
when $\lambda$ is small or null and $\sigma_d$ is small 
is possible in a sparse recovery framework. 

\begin{prop}
Consider a family $\C{L}$ of linear subspaces of $\B{R}^d$. 
Assume that $\theta_{\lambda} \in L_* \in \C{L}$ and that $\lVert 
\theta_{\lambda} \rVert \leq A$, a known constant. 

Consider the confidence region
\[ 
\wh{\Theta}_{\lambda} = \Bigl\{ \xi \in \B{R}^d \, : \, \bigl\lVert \bigl( 
\wh{G} + \lambda I \bigr)(\xi - \wh{\theta}_{\lambda}) \bigr\rVert 
\leq \epsilon \lVert \xi \rVert + \eta, \, \lVert \xi \rVert \leq A 
\Bigr\}. 
\] 
Define the model selector
\begin{align*}
\wh{\C{L}} & = \Bigl\{ \; L \in \C{L} \, : \, \wh{\Theta}_{\lambda} 
\cap L \neq \varnothing \; \Bigr\},\\ 
\wh{L} & \in \arg \max \Bigl\{ \; \wh{\sigma}_{L} \, : \, 
L \in \wh{\C{L}} \; \Bigr\},
\end{align*}
and the estimator 
\[ 
\wt{\theta} \in \arg\min \Bigl\{ \; \lVert \xi \rVert \, : \, 
\xi \in \wh{\Theta}_{\lambda} \cap \wh{L} \; \Bigr\}.
\] 
Define 
\[ 
\sigma_* = \inf \Bigl\{ \; \sigma_{L + \B{R} \theta_{\lambda}} \, : \,
L \in \C{L}, \; \sigma_{L} \geq \sigma_{L_*} - 2 \epsilon \; \Bigr\}
\] 
Under the event described by equation \myeq{eq:8.3.2}, 
\[
\bigl( \sigma_* + \lambda \bigr) \lVert \wt{\theta} - \theta_{\lambda} 
\rVert \leq \bigl\lVert (G + \lambda I) (\wt{\theta} - \theta_{\lambda} 
) \bigr\rVert \leq 2 \bigl( \epsilon \lVert \wt{\theta} \rVert + \eta \bigr)
\leq 2 \bigl( \epsilon A + \eta \bigr), 
\]
and 
\[ 
R_{\lambda}(\wt{\theta}) - R_{\lambda}( \theta_{\lambda}) 
\leq \frac{4}{\lambda + \sigma_*} \bigl( \epsilon \lVert \wt{\theta} 
\rVert + \eta \bigr)^2 \leq \frac{4}{\lambda + \sigma_*} \bigl( 
\epsilon A + \eta \bigr)^2.
\] 
\end{prop}
\begin{proof}
Since $\wt{\theta} \in \wh{\Theta}_{\lambda}$, 
\[ 
\lVert (G + \lambda I)(\wt{\theta} - \theta_{\lambda} ) 
\rVert \leq 2 \bigl( \epsilon \lVert \wt{\theta} \rVert + \eta \bigr) 
\leq 2 \bigl( \epsilon A + \eta \bigr).
\] 
On the other hand, 
\[ 
\lVert (G + \lambda I)(\wt{\theta} - \theta_{\lambda} ) \rVert
\geq \lVert \pi_{\wh{L} + \B{R} \theta_{\lambda}} (G 
+ \lambda I)(\wt{\theta} - \theta_{\lambda}) \rVert \geq 
\bigl( \sigma_{\wh{L} + \B{R} \theta_{\lambda}} + \lambda ) 
\lVert \wt{\theta} - \theta_{\lambda} \rVert.
\] 
Moreover, $L_* \in \wh{\C{L}}$, since $\theta_{\lambda} \in 
\wh{\Theta}_{\lambda} \cap L_* \neq \varnothing$. Thus 
\[ 
\sigma_{\wh{L}} \geq \wh{\sigma}_{\wh{L}} - \epsilon \geq 
\wh{\sigma}_{L_*} - \epsilon \geq \sigma_{L_*} - 2 \epsilon, 
\] 
so that $\ds \sigma_{\wh{L} + \B{R} \theta_{\lambda}} \geq \sigma_*$, 
according to the definition of $\sigma_*$, implying that 
\[ 
(\sigma_* + \lambda) \lVert \wt{\theta} - \theta_{\lambda} \rVert 
\leq \lVert (G + \lambda I) ( \wt{\theta} - \theta_{\lambda} ) \rVert 
\leq 2 \bigl( \epsilon \lVert \wt{\theta} \rVert + \eta \bigr), 
\] 
and consequently that 
\[ 
R_{\lambda} ( \wt{\theta} ) - R_{\lambda} ( \theta_{\lambda}) 
\leq \lVert \wt{\theta} - \theta_{\lambda} \rVert \; \lVert 
( G + \lambda I) ( \wt{\theta} - \theta_{\lambda} ) \rVert 
\leq \frac{4}{\sigma_* + \lambda} \bigl( \epsilon \lVert \wt{\theta} 
\rVert + \eta \bigr)^2. 
\]
\end{proof}

Remark that the constant $\sigma_*$ is defined in terms 
of restricted eigenvalues of the Gram matrix, a concept 
that has been used by other authors, for example in \cite{Tsybakov}, 
to set the conditions
of sparse recovery. 

In the case of nested models, 
we can replace the constant $\sigma_*$ with a simpler one, 
as in the following proposition. 

\begin{prop}
Consider a nested family of linear subspaces of $\B{R}^d$
\[
\C{L} = \bigl\{ L_1 \subset L_2 
\subset \cdots \subset L_K \bigr\}.
\]
Assume that $\theta_{\lambda} \in L_* \in \C{L}$, where $L_*$ is unknown, 
and that $\lVert \theta_{\lambda} \rVert \leq A$, where $A$ 
is known. 
Consider the confidence region 
\[ 
\wh{\Theta}_{\lambda} = \Bigl\{ \xi \in \B{R}^d \, : \, 
\bigl\lVert \bigl( \wh{G} + \lambda I \bigr) ( \xi - \wh{\theta}_{\lambda} 
) \bigr\rVert \leq \epsilon \lVert \xi \rVert + \eta, \; \lVert \xi \rVert 
\leq A \Bigr\}. 
\]
Define the model selector
\begin{align*}
\wh{k} & = \arg \min \Bigl\{ \; j \, : \, 
\wh{\Theta}_{\lambda} \cap L_j \neq \varnothing \; \Bigr\}, \\ 
\wh{L} & = L_{\wh{k}},
\end{align*}
and the estimator 
\[ 
\wt{\theta} \in \arg \min \, \Bigl\{ \; \lVert \xi \rVert \, : \, 
\xi \in \wh{\Theta}_{\lambda} \cap \wh{L} \; \Bigr\}. 
\] 
Under the event described by equation \myeq{eq:8.3.2}, 
\[ 
\bigl( \sigma_{L_*} + \lambda \bigr ) \lVert \wt{\theta} 
- \theta_{\lambda} \rVert \leq \bigl\lVert (G + \lambda I) 
( \wt{\theta} - \theta_{\lambda} ) \bigr\rVert \leq 
2 \bigl( \epsilon \lVert \wt{\theta} \rVert + \eta \bigr) 
\leq 2 ( \epsilon A + \eta ),  
\] 
and 
\[ 
R_{\lambda}(\wt{\theta}) - R_{\lambda}(\theta_{\lambda}) 
\leq \frac{4}{ \lambda + \sigma_{L_*}} \bigl( \epsilon \lVert \wt{\theta} 
\rVert + \eta \bigr)^2 \leq \frac{4}{\lambda + \sigma_{L_*}} 
\bigl( \epsilon A + \eta \bigr)^2.  
\] 
\end{prop}
\begin{proof}
As in the previous proposition, 
\[
\wt{\theta} \in \wh{\Theta}_{\lambda}, 
\]
so that $\lVert (G + \lambda I) ( \wt{\theta} - \theta_{\lambda}) \rVert 
\leq 2 \bigl( \epsilon \lVert \wt{\theta} \rVert + \eta \bigr)$. 
Moreover $L_* \cap \wh{\Theta}_{\lambda} \neq \varnothing$, 
so that $\wh{L} \subset L_*$, implying that 
\[
\bigl( \sigma_* + \lambda \bigr) \lVert \wt{\theta} - \theta_{\lambda} 
\rVert \leq \lVert \pi_{L_*} (G + \lambda I) ( \wt{\theta} 
- \theta_{\lambda} ) \rVert \leq \lVert (G + \lambda I) (\wt{\theta} 
- \theta_{\lambda}) \rVert
\]
and that $\ds R_{\lambda}(\wt{\theta}) - R_{\lambda}(\theta_{\lambda}) 
\leq \frac{4}{\sigma_* + \lambda} \bigl( \epsilon \lVert \wt{\theta} 
\rVert + \eta \bigr)^2$. 
\end{proof}

\bibliographystyle{imsart-nameyear}
\bibliography{Gram}

\end{document}